\documentclass[final,12pt]{colt2026} 

\usepackage{booktabs} 
\usepackage{algorithm}
\usepackage{algorithmic}
\usepackage{wrapfig}
\usepackage{setspace}

\newtheorem{setting}[theorem]{Setting}

\usepackage{amsfonts,amssymb,stmaryrd,bbm }

\newcommand{\bfg}{\mathbf{g}} \newcommand{\bfh}{\mathbf{h}}

 \newcommand{\E}{\mathcal{E}} \newcommand{\F}{\mathcal{F}}
  
\newcommand{\J}{\mathcal{J}} \newcommand{\K}{\mathcal{K}} 
  \newcommand{\Oc}{\mathcal{O}}
  \newcommand{\R}{\mathcal{R}}
  
  \newcommand{\X}{\mathcal{X}}

 \newcommand{\EE}{\mathbb{E}}

\newcommand{\PP}{\mathbb{P}}  \newcommand{\RR}{\mathbb{R}}

\renewcommand{\d}{\mathrm{d}} 

\newcommand{\e}{\varepsilon}

\newcommand{\KL}{\mathrm{KL}}
\newcommand{\trans}{\mathrm{trans}}
\title[Fast and Large-Scale UOT via its Semi-Dual and Adaptive Gradient Methods]{Fast and Large-Scale Unbalanced Optimal Transport via its Semi-Dual and Adaptive Gradient Methods}
\usepackage{times}
\usepackage[toc,page,header]{appendix}
\usepackage{minitoc}
\usepackage[utf8]{inputenc}

\renewcommand \thepart{}
\renewcommand \partname{}

\coltauthor{%
 \Name{Ferdinand Genans} \Email{genans.ferdinand@gmail.com}\\
 \addr  LPSM, Sorbonne Université, Paris, France
}

\begin{document}

\maketitle

\begin{abstract}%
Unbalanced Optimal Transport (UOT) has emerged as a robust relaxation of standard Optimal Transport, particularly effective for handling outliers and mass variations. However, scalable algorithms for UOT, specifically those based on Gradient Descent (SGD), remain largely underexplored. In this work, we address this gap by analyzing the semi-dual formulation of Entropic UOT and demonstrating its suitability for adaptive gradient methods. While the semi-dual is a standard tool for large-scale balanced OT, its geometry in the unbalanced setting appears ill-conditioned under standard analysis. Specifically, worst-case bounds on the marginal penalties using $\chi^2$ divergence suggest a condition number scaling with $n/\varepsilon$, implying poor scalability. In contrast, we show that the local condition number actually scales as $\mathcal{O}(1/\varepsilon)$, effectively removing the ill-conditioned dependence on $n$. Exploiting this property, we prove that SGD methods adapt to this local curvature, achieving a convergence rate of $\mathcal{O}(n/\varepsilon T)$ in the stochastic and online regimes, making it suitable for large-scale and semi-discrete applications. Finally, for the full batch discrete setting, we derive a nearly tight upper bound on local smoothness depending solely on the gradient. Using it to adapt step sizes, we propose a modified Adaptive Nesterov Accelerated Gradient (ANAG) method  on the semi-dual functional and prove that it achieves a local complexity of $\mathcal{O}(n^2\sqrt{1/\varepsilon}\ln(1/\delta))$.
\end{abstract}

\section{Introduction}
\label{sec:intro}

Optimal Transport (OT) has firmly established itself as a fundamental tool in machine learning and statistics \citep{peyre2019computational}, offering a geometrically meaningful way to compare probability distributions. Its applications span a vast landscape, including domain adaptation \citep{courty2014domain}, generative modeling \citep{an2019ae, li2023dpm}, and biological data analysis \citep{schiebinger2019optimal}. The widespread adoption of OT is largely attributable to the computational breakthrough of entropic regularization \citep{cuturi2013sinkhorn}, which enabled the use of efficient Sinkhorn-type scaling algorithms.

Building upon this foundation, Unbalanced OT (UOT) has emerged as a flexible generalization specifically designed to handle scenarios involving outliers, mass variations, or partial matching. By relaxing the strict mass conservation constraints of standard OT using $\varphi$-divergences, UOT accommodates datasets with varying total masses \citep{liero2018optimal, chizat2018unbalanced}.

Despite the success of UOT, its algorithmic landscape remains heavily skewed toward generalized Sinkhorn methods \citep{chizat2018scaling, sejourne2023unbalanced}. This is in stark contrast to balanced OT, where semi-dual formulations have become the standard for large-scale \citep{genevay2016stochastic, seguy2017large} and semi-discrete applications \citep{MerigotNewtonSemiDiscrete}. Surprisingly, semi-dual based algorithms have neither been widely proposed nor analyzed for the UOT problem. The notable exceptions are the works of \citet{vacher2023semi}, where the statistical properties of unregularized continuous UOT are analyzed via its semi-dual formulation, and \citet{choi2023generative}, who successfully used SGD on the unregularized semi-dual for generative model training. However, both the optimization geometry of the regularized semi-dual and the theoretical convergence guarantees of SGD schemes in this setting remain largely unexplored.

This gap likely stems from the theoretical "stiffness" of the unbalanced formulation. Unlike the Sinkhorn algorithm, which sees its computational complexity improve from $\mathcal{O}(n^2/\varepsilon^2)$ to $\mathcal{O}(n^2/\varepsilon)$ in the unbalanced setting, gradient-based methods on the semi-dual appear to suffer from poor conditioning. Furthermore, the standard use of Kullback-Leibler (KL) marginal penalties is often taken for granted, as it is a prerequisite for the coordinate-descent updates of Sinkhorn. This hegemony has likely overshadowed the exploration of alternative divergences for Entropic UOT. In this work, we demonstrate that employing the $\chi^2$ divergence on the target measure is actually a key factor in mitigating ill-conditioning, rendering the problem tractable for first-order methods.

\textbf{Contributions. } We challenge the perspective that UOT is ill-suited for gradient methods by providing a comprehensive analysis of the Entropic UOT semi-dual geometry. Our contributions are threefold:

\noindent $\bullet$  \textbf{Theoretical Analysis of the Entropic Semi-Dual Geometry:} We provide a thorough analysis of the Entropic UOT semi-dual. Our key observation is that the local condition number at the optimizer scales strictly as $\mathcal{O}(1/\varepsilon)$, effectively removing the dependence on the problem size $n$. Furthermore, we identify key properties enabling gradient schemes to adapt to this favorable geometry, specifically establishing the generalized self-concordance, global smoothness bounds, and $(L_0, L_1)$-type smoothness of the semi-dual.

\noindent $\bullet$ \textbf{Stochastic Regime:} We analyze the Projected Averaged SGD (PASGD) for the UOT semi-dual. We prove that PASGD naturally adapts to the benign local geometry, achieving a convergence rate of $\mathcal{O}(n/\varepsilon T)$. This yields a lightweight solver suitable for massive datasets.

\noindent $\bullet$ \textbf{Deterministic Regime:} For the full-batch discrete setting, we leverage a tight, data-dependent upper bound on the local smoothness. We utilize this bound to design an Adaptive Nesterov Accelerated Gradient (NAG) method. Unlike standard acceleration, which relies on conservative global constants, our method adjusts its step size dynamically, achieving a local complexity of $\mathcal{O}(n^2\sqrt{1/\varepsilon}\ln(1/\delta))$.

\section{Background on Unbalanced Optimal Transport}
\label{sec:background_uot}

Let $(\mathcal X,d_{\mathcal X})$ and $(\mathcal Y,d_{\mathcal Y})$ be Polish spaces.
We consider finite nonnegative measures $\mu\in\mathcal M_+(\mathcal X)$ and $\nu\in\mathcal M_+(\mathcal Y)$,
and a continuous non-negative cost $c:\mathcal X\times\mathcal Y\to\mathbb R$.
For a coupling $\pi\in\mathcal M_+(\mathcal X\times\mathcal Y)$, we denote its marginals by
$\pi_1\in\mathcal M_+(\mathcal X)$ and $\pi_2\in\mathcal M_+(\mathcal Y)$. 

\textbf{Unbalanced OT with entropic regularization.}
Unbalanced optimal transport (UOT) compares $\mu$ and $\nu$ while allowing mass variation:
instead of enforcing $\pi_1=\mu$ and $\pi_2=\nu$, it penalizes marginal mismatch.
With entropic regularization $\varepsilon>0$ and penalty weights $\rho_1,\rho_2>0$:
\begin{equation}
\label{eq:entropic_uot_primal}
\mathrm{UOT}^{\rho_1, \rho_2}_{\varepsilon,c}(\mu,\nu) := \min_{\pi\in\mathcal M_+(\mathcal X\times\mathcal Y)}
\int_{\mathcal X\times\mathcal Y} c\d\pi
+\varepsilon\,\KL(\pi\mid \mu\otimes \nu)
+\rho_1\, D_{1}(\pi_1\mid \mu)
+\rho_2\, D_{2}(\pi_2\mid \nu)\ .
\end{equation}
The marginal penalties $D_1$ and $D_2$ are chosen as Csisz\'ar ($f$-)divergences \citep{csiszar1967information,ali1966general}, which are of the form
\[
D_{\varphi}(\alpha\mid \beta)
:= 
\displaystyle \int \varphi\!\left(\frac{d\alpha}{d\beta}\right)\, d\beta,\quad \text{for } \alpha\ll \beta, \quad \text{otherwise :} +\infty.
\]
for $\varphi:[0,\infty)\to\mathbb R\cup\{+\infty\}$ convex with $\varphi(1)=0$.  The usual choices in OT are 
(i) KL, with $\varphi_{\KL}(t)=t\ln t-t+1$ and $D_{\varphi_{\KL}}=\KL$,
and (ii) the quadratic (Pearson $\chi^2$) divergence, with $\varphi_{\chi^2}(t)=\tfrac12(t-1)^2$ and $D_{\varphi_{\chi^2}}=D_{\chi^2}$.

\noindent The joint term $\KL(\pi\mid\mu\otimes\nu)$ is the standard entropic regularization, which makes the objective smoother
and enables efficient iterative solvers (Sinkhorn-type schemes in OT \citep{cuturi2013sinkhorn} and in UOT \citep{chizat2018scaling}).
Finally, $\varepsilon\to 0$ recovers the unregularized UOT objective, while $\rho_1,\rho_2\to+\infty$ formally enforces balanced marginals and recovers (entropic) OT.

\textbf{Entropic dual.}
Entropic UOT admits a dual formulation over measurable potentials $f:\mathcal X\to\mathbb R$ and $g:\mathcal Y\to\mathbb R$:
\begin{align*}
    \mathrm{UOT}^{\rho_1, \rho_2}_{\varepsilon,c}(\mu,\nu) = \sup_{f,g}\Bigg\{\!\!
    -\varepsilon\!\!\int\!\!
    \exp\left(\tfrac{f+g-c}{\varepsilon}\right)\d\mu\d\nu-\rho_1\!\!\int_{\mathcal{X}}\!\varphi_1^c\left(-\tfrac{f}{\rho_1}\right)\d\mu
    -\rho_2\!\!\int_{\mathcal Y}\!\varphi_2^c\left(-\tfrac{g}{\rho_2}\right)\d\nu\Bigg\},
\end{align*}
up to additive constants independent of $(f,g)$.
Here $\varphi^c:\mathbb R\to\mathbb R\cup\{+\infty\}$ is the convex conjugate $ \varphi^c(s)=\sup_{t\ge 0}\{st-\varphi(t)\}$.
In particular:
\begin{align*}
    \varphi_{\KL}^c(s)&=e^{s}-1,\qquad s\in\mathbb R, \\
    \varphi_{\chi^2}^c(s)&= s+\tfrac12 s^2 \quad \text{ if } s\ge -1\ ,   \quad \text{otherwise }   -\tfrac{1}{2}\ .
\end{align*}

\paragraph{Dual optimizers and induced coupling.}
The dual potentials explicitly parameterize the optimal coupling: if $(f^\star,g^\star)$ maximizes the dual,
then $\pi^\star\ll \mu\otimes\nu$ and
\begin{equation}
\label{eq:pi_from_dual}
\frac{d\pi^\star}{d(\mu\otimes\nu)}(x,y)
=
\exp\left(\frac{f^\star(x)+g^\star(y)-c(x,y)}{\varepsilon}\right).
\end{equation}

\section{Entropic UOT Semi-Dual: Derivation and Properties}
\label{sec:semi-dual}

Having established the primal and dual formulations, we now specialize to the, at least, semi-discrete paradigm. In this setting, the target measure $\nu$ is discrete (or has been discretized), while the source measure $\mu$ remains abstract (continuous or discrete). The choice of the target measure being discrete, rather than the source one, is arbitrary here. 
\vspace{-0.3em}
\begin{setting}\label{setting::semi_discrete}
    We assume $\nu$ is a discrete positive measure supported on $n$ points $\{y_1, \dots, y_n\} \subset \mathcal{Y}$:
    \[
    \nu = \sum_{j=1}^n \beta_j \delta_{y_j}, \quad \text{with } \beta_j > 0.
    \]
    To ensure stability, we assume the weights are balanced relative to the resolution $n$. Specifically, there exist constants $b/B \approx 1$ such that for all $j$, $b/n \le \beta_j \le B/n$.
\end{setting}

\subsection{Semi-Dual Formulation and Gradient}

The semi-dual functional is derived by explicitly solving the maximization of the dual objective with respect to the source potential $f$. This reduces the problem to an unconstrained optimization over the target potential vector $\mathbf{g} = (g_1, \dots, g_n) \in \mathbb{R}^n$. Before introducing the semi-dual functional, we define the following auxiliary quantities for any $x \in \mathcal{X}$:
\begin{align*}
B_j(x,\mathbf{g}) := \beta_j \exp\left(\tfrac{g_j - c(x,y_j)}{\varepsilon}\right), 
Z(x,\mathbf{g}) := \sum_{j=1}^n B_j(x,\mathbf{g})\ ,  \quad w_j(x, \mathbf{g}) := \tfrac{B_j(x, \mathbf{g})}{Z(x, \mathbf{g})}.
\end{align*}

\noindent We maintain a fixed target penalty $D_2 = D_{\chi^2}$ but consider two standard options for the source penalty $D_1$:

\noindent $\bullet$ \textbf{(KL-source)} $D_1 = \text{KL}$, allowing for a closed-form elimination of $f$.

\noindent $\bullet$ \textbf{($\chi^2$-source)} $D_1 = D_{\chi^2}$, which involves the Lambert $W$ function.

\noindent To make the presentation easier, in the core of the paper, we fix $D_1 = \KL$, since it has a slightly easier form to present and derive the theorems. However, all the results here are the same for  $D_1 = \chi^2$ , up to constants. The proofs for the case  $D_1 = \chi^2$ are also in the appendix. 

\begin{proposition}[Semi-Dual Objective and Gradient]
\label{prop:unified_semi_dual}(Proof in Appendix \ref{appendix:semi_dual}.)
With $\alpha := \frac{\varepsilon}{\varepsilon + \rho_1}$, the  semi-dual objective $\mathcal{J}: \mathbb{R}^n \to \mathbb{R}$ is
\begin{equation}
\label{eq:semidual_unified}
\mathcal{J}(\mathbf{g}) = (\rho_1 + \varepsilon) \int_{\mathcal{X}} Z(x,\mathbf{g})^\alpha  \, d\mu(x) + \sum_{j=1}^n \beta_j \left( \tfrac{g_j^2}{2\rho_2} - g_j \right).
\end{equation}
Its gradient with respect to the $k$-th component is given by:
\begin{equation}
\label{eq:grad_unified}
\nabla_k \mathcal{J}(\mathbf{g}) = \int_{\mathcal{X}}   Z(x,\mathbf{g})^\alpha  \, w_k(x, \mathbf{g}) \, d\mu(x) + \tfrac{\beta_k}{\rho_2}g_k - \beta_k.
\end{equation}
\end{proposition}

\noindent \textbf{Remark:} While we allow different divergences for the source measure, fixing the $\chi^2$ divergence for the target marginal is compulsory to have a data independent strong convexity. For instance, this is not (always) the case when $D_2 = \KL$, we refer to Section \ref{subsec:glob_and_loc_curv} for more details.

\subsection{The First Order Condition Keystone}

The efficiency of gradient-based algorithms is governed by the geometry of the objective function. While the strong convexity of $\mathcal{J}$ is evident from the target $\chi^2$ divergence term, the smoothness is more subtle. We show here that by analyzing the smoothness locally, we can derive a much tighter conditioning bound. To do so, we identify the \textit{transport} part of the objective, which is the key component of our analysis:
\begin{align*}
\mathcal{J}_{\mathrm{trans}}(\mathbf{g}) = (\rho_1 + \varepsilon) \int_{\mathcal{X}} Z(x,\mathbf{g})^\alpha \, d\mu(x);  \quad    
[\nabla \mathcal{J}_{\mathrm{trans}}(\mathbf{g})]_k = \int_{\mathcal{X}} Z(x,\mathbf{g})^\alpha \, w_k(x, \mathbf{g}) \, d\mu(x)\ .
\end{align*}
From it, we state our key smoothness result, which holds for both source divergences.

\begin{theorem}[Smoothness Bound via Gradient Transport]\label{thm:smoothness_unified} 
(Proof in Appendix \ref{app::proof_thm_smoothness_unified}.) 
For all $\mathbf{g} \in \mathbb{R}^n$, the operator norm of the Hessian satisfies:
\begin{equation}
    \|\nabla^2 \mathcal{J}(\mathbf{g})\|_{\mathrm{op}} \le \frac{1}{\varepsilon} \|\nabla \mathcal{J}_{\mathrm{trans}}(\mathbf{g})\|_\infty + \frac{\beta_{\max}}{\rho_2}.
\end{equation}
\end{theorem}

\noindent This theorem is our \textit{keystone}: it transforms the problem of bounding curvature into the problem of bounding the gradient. This link will be further leveraged in the next section. 

The next proposition is straightforward to derive yet fundamental; it establishes that the optimal potentials are naturally confined to a region where the geometry is well-behaved.

\begin{proposition}[First-Order Optimality]
\label{prop:optimality_constraints}
At the global minimizer $\mathbf{g}^\star$, the condition $\nabla \mathcal{J}(\mathbf{g}^\star) = 0$ implies $[\nabla \mathcal{J}_{\mathrm{trans}}(\mathbf{g}^\star)]_k = \beta_k ( 1 - g_k^\star / \rho_2 )$. Since $\nabla \mathcal{J}_{\mathrm{trans}} \ge 0$, the optimizer satisfies the automatic box constraint:
\begin{equation}
    g_k^\star \le \rho_2, \quad \forall k \in \{1, \dots, n\}.
\end{equation}
\end{proposition}

Building on this, we define the feasible region $\mathcal{K}$, which will serve as the focal point for our subsequent analysis and in our algorithms, where we will use $\delta = 0.1$ by default. 
\begin{equation*}
\mathcal{K}_\delta := \{ \mathbf{g} \in \mathbb{R}^n \mid g_k \le \rho_2 + \delta, \forall k \},
\end{equation*}
where $\delta > 0$ is a small margin ensuring $\mathbf{g}^\star \in \mathrm{int}(\mathcal{K})$. By focusing on $\mathcal{K}$, we can transition from point-wise optimality to a more global understanding of the objective's geometry. In particular, in the following, we leverage this set to derive dimension-independent bounds on the curvature. 

\subsection{Global and Local Curvature}\label{subsec:glob_and_loc_curv}

\begin{lemma}[Uniform Gradient Bound and Smoothness]
\label{lem:global_grad_bound} (Proof in Appendix \ref{app::proof:lem:global_grad_bound}.) 
  On $\mathcal{K}$, the $L_1$-norm of the transport gradient is uniformly bounded:  $\| \nabla \mathcal{J}_{\mathrm{trans}}(\mathbf{g}) \|_1 \le C_{\mathrm{bound}}$, where:
\begin{equation}
    C_{\mathrm{bound}}^{\KL} :=
    \mu(\mathcal{X}) \|\nu\|_1^\alpha \exp\left( \frac{\rho_2 + \delta}{\rho_1 + \varepsilon} \right) 
\end{equation}
Consequently, the Hessian is bounded on $\mathcal{K}$, and $\mathcal{J}$ is $L$-smooth with $L = \mathcal{O}(1/\varepsilon)$.
\end{lemma}

This result demonstrates that the objective remains "flat" enough for stable optimization even as the resolution $n \to \infty$, when we are in $\K$.  Notably, the bound is independent of $n$ since $\|\nu\|_1 \leq B$. However, it grows exponentially with the margin $\delta$, illustrating that $\mathcal{J}$ is not globally smooth; its curvature is only controlled near the optimizer. Finally, we establish the conditioning of the problem, where the condition number for an $L$-smooth and $\gamma$-strongly convex objective is defined as
\[
    \kappa := \frac{L}{\gamma}\ . 
\]

\begin{corollary}[Local Conditioning]
\label{cor:local_conditioning}
The objective $\mathcal{J}$ is $\frac{\beta_{\min}}{\rho_2}$-strongly convex on $\mathbb{R}^n $. At the optimizer $\mathbf{g}^\star$, the local condition number $\kappa$ satisfies:
\begin{equation}
    \kappa(\nabla^2 \mathcal{J}(\mathbf{g}^\star)) \le \frac{\beta_{\max}}{\beta_{\min}} \left( 1 + \frac{\max_k \{\rho_2 - g_k^\star\}}{\varepsilon} \right).
\end{equation}
\end{corollary}

This result highlights a crucial disconnect between the global worst-case analysis derived in Lemma \ref{lem:global_grad_bound} and the local reality of the problem. Even on $\mathcal{K}$, a global bound on $L$ is of the form $\Oc(1/\varepsilon)$, leading to a condition number of $\Oc(N\rho_2/\varepsilon)$. In contrast, Corollary~\ref{cor:local_conditioning} shows that locally, the conditioning depends only on the regularization ratio $\rho_2/\varepsilon$ and on the mass balance ratio $\beta_{\max}/\beta_{\min}$, which we assume is $\approx 1$. Crucially, this local condition number is independent of $n$. This observation motivates  the use of first-order methods that can adapt to the local curvature and converge at a rate independent of the problem size $n$.

The key to unlocking these local acceleration guarantees lies in our proof of generalized self-concordance for the semi-dual. While this property has been successfully applied to logistic regression \cite{bach2010self, bach2014adaptivity} and discrete optimal transport \cite{sun2019generalized}, establishing it in our context allows us to strictly control the change of the Hessian locally. We utilize this property in both our semi-discrete and discrete settings to derive rates that depend on the favorable local geometry rather than global worst-case bounds.

\begin{proposition}[Generalized self-concordance]\label{prop::self_concordance} (Proof in Appendix \ref{app::proof_self_conc}.)
    The semi-dual $\J$ is generalized self-concordant. That is, for $M = \frac{2 + 3\alpha}{\varepsilon}$ for $\KL$ source, and $M = \frac{6}{\varepsilon}$ for $\chi^2$ , we have for any $\bfg \in \mathbb{R}^n$ and any direction $\bfh \in \mathbb{R}^n$:
    \[
        \big| \nabla^3 \J(\bfg)[\bfh, \bfh, \bfh] \big| \le M \|\bfh\|_\infty \, \langle \bfh, \nabla^2 \J(\bfg) \bfh \rangle.
    \]
\end{proposition}

\paragraph{The Necessity of the Target $\chi^2$ Divergence.}
\label{subsec::kl-kl}
As highlighted in Remark~\ref{prop:unified_semi_dual}, the $\chi^2$ target penalty offers a decisive geometric advantage over the standard KL-KL formulation. In the KL-KL setting, the diagonal terms of the dual Hessian scale with $e^{-g/\rho_2}$. Consequently, the curvature becomes highly anisotropic and vanishes exponentially as $g \to \infty$ (a phenomenon linked to mass destruction). This lack of uniform strong convexity severely complicates the analysis of accelerated algorithms. In stark contrast, the $\chi^2$ penalty ensures a constant curvature of $1/\rho_2$ and guarantees global strong convexity. While the KL-KL geometry remains manageable in the batch discrete setting (discussed in Section~\ref{subsec:full_discrete}), the $\chi^2$ formulation provides the structural stability required for our results.

\section{Adaptive Gradient Descent on the Semi-Dual}
\label{sec:algo}

\subsection{Large-Scale and Semi-Discrete Settings}
\label{subsec:sgd_semidiscrete}

We first address the setting where the source measure $\mu$ is continuous, or discrete with a cardinality large enough to make full-batch processing impossible. In this regime, Stochastic Gradient Descent (SGD) is the natural algorithmic choice. We demonstrate that a simple Projected Averaged SGD (PASGD) scheme achieves efficient convergence, escapes the worst-case analysis, and leverages the local smoothness of the semi-dual. 

\paragraph{Unbiased Gradient Estimator.}
Assume $\mu(\mathcal{X}) < \infty$ and let $X_1, \dots, X_{m_b} \overset{\text{i.i.d.}}{\sim} \mu / \mu(\mathcal{X})$ be a batch of samples drawn from the normalized source measure.
Using the gradient formulation \eqref{eq:grad_unified}, we define the stochastic gradient estimator $\widehat{\nabla} \mathcal{J}(\mathbf{g})$ component-wise for $k \in \{1, \dots, n\}$:
\begin{equation}
\label{eq:stoch_grad_estimator}
\widehat{\nabla}_k \mathcal{J}(\mathbf{g})
:=
\frac{\mu(\mathcal{X})}{m_b} \sum_{i=1}^{m_b} Z(X_i, \mathbf{g})^\alpha \, w_k(X_i, \mathbf{g}) + \frac{\beta_k}{\rho_2}g_k - \beta_k.
\end{equation}
By the linearity of expectation and the identity $\mathbb{E}_{X \sim \mu/\mu(\mathcal{X})}[h(X)] = \frac{1}{\mu(\mathcal{X})} \int h \, d\mu$, it is immediate that this estimator is unbiased: $\mathbb{E}[\widehat{\nabla} \mathcal{J}(\mathbf{g})] = \nabla \mathcal{J}(\mathbf{g})$.

\paragraph{Complexity.}
Computing this estimator for the full vector $\mathbf{g} \in \mathbb{R}^n$ requires only $\mathcal{O}(m_b \cdot n)$ operations. This linear complexity in $n$ makes the approach scalable to large target supports.

\paragraph{The Online Regime.}
A significant advantage of this stochastic formulation is its direct applicability to the \textit{online} setting \cite{hazan2016introduction}. Unlike batch methods that require repeated passes over a fixed dataset, our approach naturally handles streaming data where samples $X_t$ arrive sequentially. In this regime, the algorithm effectively minimizes the population risk (the integral against $\mu$) directly, rather than the empirical risk. This is particularly valuable for semi-discrete tasks, as in generative modeling \citep{li2023dpm, choi2023generative} where storing the full history of samples is memory-prohibitive. The constant memory footprint, storing only the potential $\mathbf{g} \in \mathbb{R}^n$, remains independent of the number of samples processed.

\paragraph{Adaptivity and Convergence.}
We propose to solve the semi-dual problem using Projected Averaged SGD (PASGD). While standard SGD can suffer from oscillation in ill-conditioned settings, averaging the iterates (Polyak-Ruppert averaging) is known to statistically adapt to the local curvature, achieving optimal asymptotic rates.

We define the update rule at step $t$ with step size $\eta_t$ as:
\begin{equation}
    \mathbf{g}_{t+1} = \Pi_{\mathcal{K}}\left( \mathbf{g}_t - \eta_t \widehat{\nabla} \mathcal{J}(\mathbf{g}_t) \right), \quad \bar{\mathbf{g}}_T = \frac{1}{T} \sum_{t=1}^T \mathbf{g}_t.
\end{equation}
Here, $\Pi_{\mathcal{K}}$ denotes the projection onto the feasible set $\mathcal{K} = \{ \mathbf{g} \mid g_k \le \rho_2 + \delta \}$, which ensures the iterates remain in the region where the gradient variance and Hessian are well-behaved (as established in Section \ref{sec:semi-dual}). 

Before stating the main convergence result, we give a simple lemma, important for the analysis of SGD schemes. 

\begin{proposition}[Variance Bound of Mini-Batch Gradient]\label{lemma::variance} (Proof in Appendix \ref{app::proof_lemma::variance}.) 
Let $\widehat{\nabla} \mathcal{J}(\mathbf{g})$ be the mini-batch gradient estimator computed with batch size $m_b\geq 1$, as defined in Eq. \eqref{eq:stoch_grad_estimator}.
For any $\mathbf{g} \in \mathcal{K}$, using the uniform bound $C_{\mathrm{bound}}$ from Lemma \ref{lem:global_grad_bound}, the variance is bounded by:
\begin{equation}
    \mathbb{E}\left[ \| \widehat{\nabla} \mathcal{J}(\mathbf{g}) - \nabla \mathcal{J}(\mathbf{g}) \|_2^2 \right] 
    \le 
    \frac{4 C_{\mathrm{bound}}^2}{m_b} \ . 
\end{equation}
\end{proposition}

\begin{theorem}[Convergence of PASGD]
\label{thm:pasgd_convergence}(Proof in Appendix \ref{app:proof_pasgd}.) 
Let the step sizes be chosen as $\eta_t = C t^{-\gamma}$ with $\gamma \in (1/2, 1)$. Under Setting \ref{setting::semi_discrete} and the projection onto $\mathcal{K}$, the averaged iterate $\bar{\mathbf{g}}_T$ converges to the optimum $\mathbf{g}^\star$ in objective value with an expected error of:
\[
\mathbb{E}\left[ \mathcal{J}(\bar{\mathbf{g}}_T) - \mathcal{J}(\mathbf{g}^\star) \right] = \mathcal{O}\left( \frac{n\rho_2^2}{\varepsilon T} \right).
\]
\end{theorem}
 \noindent We give here a sketch of proof, to illustrate where and how we are able to leverage the local condtionning near the optimum.
\noindent \textit{Sketch of Proof.} Classical results \cite{polyak1992acceleration, gadat2017optimal} using global strong convexity and uniform gradient error bound from Lemma \ref{lem:global_grad_bound} assure, with $\mathbf{H} = \nabla^2 \J(\bfg^\star)$ and $\mathbf{\Sigma}$ the noise covariance at the optimum:
\begin{equation*}
    \mathbb{E}[\|\bar{\mathbf{g}}_T - \mathbf{g}^\star\|^2] \le \frac{\text{Tr}(\mathbf{H}^{-1}\mathbf{\Sigma}\mathbf{H}^{-1})}{Tm_b} + o(1/T).
\end{equation*}
From the global strong convexity of $\mathcal{J}$, this term scales as $\mathcal{O}(\rho_2^2 n^2/T)$.
To handle the non-global smoothness, we split the objective gap and will use the generalized self-concordance property. 

\noindent Locally, $\mathcal{J}$ is $L$-smooth with $L \propto 1/n\varepsilon$. From the generalized self-concordance, we have (see Corollary \ref{cor::self_conc_hessian}): $L(\bfg_T) \leq \exp(\tfrac{2 + 3\alpha}{\varepsilon}\|\bfg_T - \bfg^\star\|)L(\bfg^\star)$. 
Therefore
\begin{align*}
    \mathbb{E}[\mathcal{J}(\bar{\mathbf{g}}_T) - \mathcal{J}^\star] \leq \frac{L(\bfg^\star)(2+3\alpha)}{2} \mathbb{E}[\|\bar{\mathbf{g}}_T - \mathbf{g}^\star \|^2]  + \frac{C_1}{\varepsilon} \mathbb{P}(\| \bar{\mathbf{g}}_T - \mathbf{g}^\star \| \geq \varepsilon).
\end{align*}

Due to the concentration of the average SGD, with $\gamma_t \propto t^{-b}, b\in (1/2, 1)$, which gives high order moments for both the SGD and ASGD schemes, $\mathbb{P}(\| \bar{\mathbf{g}}_T - \mathbf{g}^\star \| \geq \varepsilon)$ is negligible, which concludes. \hfill $\square$

\textbf{Remark:} Observe that for this proof of adaptivity, we need $\gamma_t \propto t^{-b}$ with $b<1$. However, we refer to Appendix \ref{app::towards}, for a motivation of the study of other SGD schemes and learning rate, that could lead to an even better complexity, by exploiting even more the good local conditioning.

\subsubsection{Numerical Experiments: Semi-Discrete Setting}

We validate our theoretical findings on a synthetic 10-dimensional mixture of Gaussians with 4 modes. The target $\nu$ is discrete with $n=2,000$ uniformly weighted points. To evaluate convergence, we compute a high-precision ground truth $\mathbf{g}^\star$ using the deterministic Adaptive NAG solver (Section~\ref{subsec:full_discrete}) run to machine precision, with $n' = 10,000$ points for the source measure. All results report the average of 20 independent runs; variance was negligible and is omitted for clarity.

\noindent\textbf{Efficiency of PASGD.} We reaffirm here that scaling the learning rate by the inverse of the strong convexity yields optimal results. Figure~\ref{fig:sgd_comparison} compares standard SGD against Projected Averaged SGD (PASGD) using a step size decay of $\eta_t = C \frac{n}{\rho_2}(t+ 1/\varepsilon)^{-2/3}$ for varying scales $C$. Here, $\varepsilon = 0.01$.  For $b = 2/3$, the choice $C=1$, which corresponds to the natural scale dictated by the strong convexity of $\mathcal{J}$, achieves the best performance. 

\begin{figure}[h]
    \centering
    \includegraphics[width=0.35\linewidth]{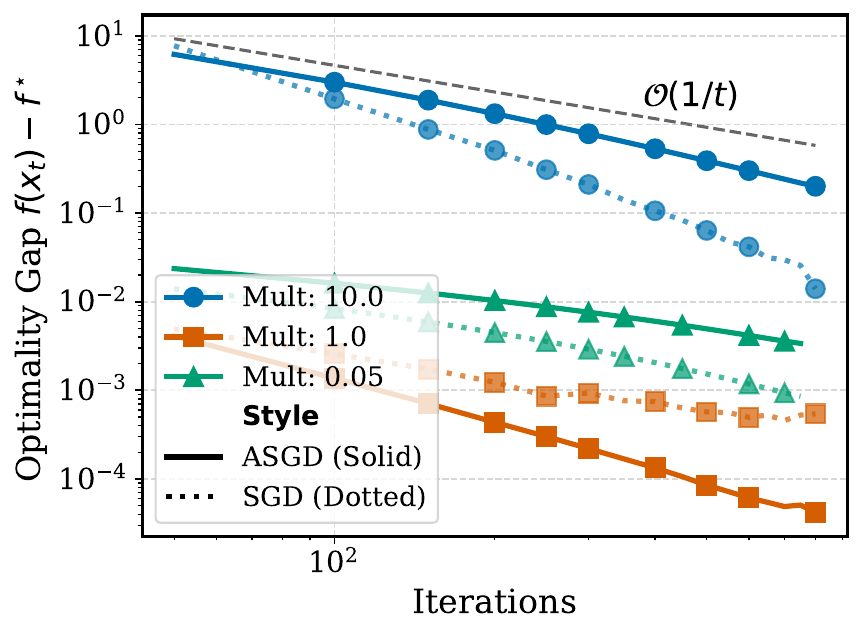}
    \caption{\textbf{PASGD vs. SGD.} Convergence of the objective gap on a semi-discrete UOT problem ($n=2000$), with $\eta_t = C \frac{n}{\rho_2}(t+ 1/\varepsilon)^{-2/3}$,  $C \in \{0.05, 1, 10\}$. PASGD confirms the $\mathcal{O}(1/T)$ rate and shows superior performance compared to SGD. }
    \label{fig:sgd_comparison}
\end{figure}

\textbf{Impact of Regularization.} Figure~\ref{fig:eps_dependence} examines the sensitivity of the algorithm to the regularization parameter $\varepsilon$. While optimal theoretical bounds suggest the possibility of eliminating the dependency on $\varepsilon$, we observe that practical performance retains some sensitivity. Specifically, decreasing $\varepsilon$ impacts the convergence of the objective function, which scales with roughly $1/\varepsilon$. However, importantly, the squared distance to the optimum $\|\bar{\mathbf{g}}_t - \mathbf{g}^\star\|^2$ exhibits a much milder dependence on $\varepsilon$, demonstrating significant robustness in parameter space.

\begin{figure}[h]
    \centering
    \includegraphics[width=0.5\linewidth]{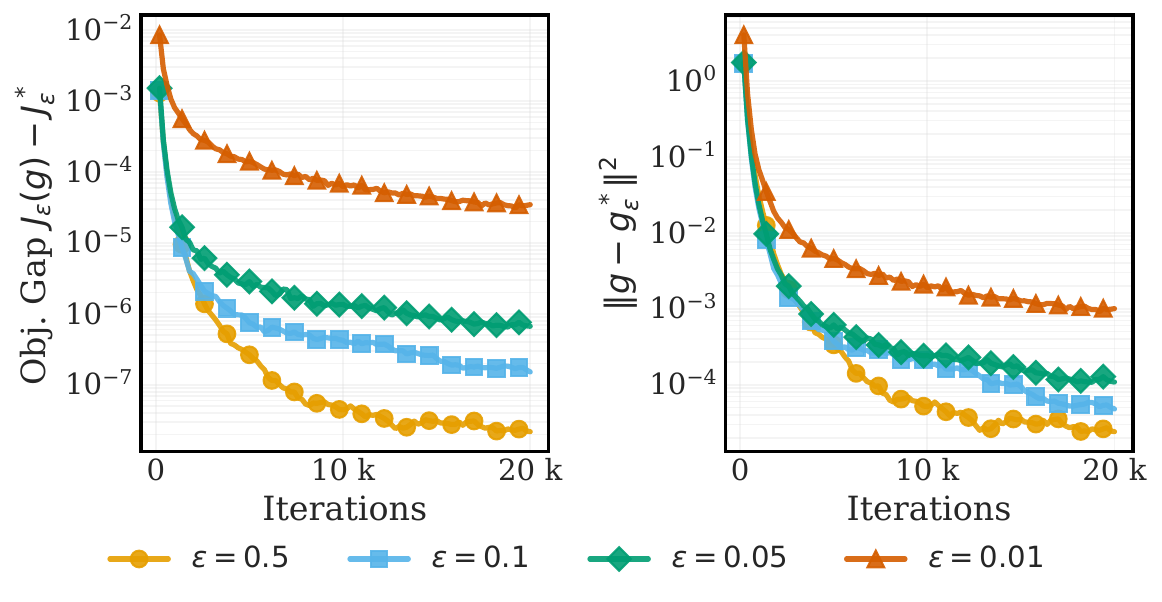}
    \caption{\textbf{Effect of $\varepsilon$.} Convergence profiles for varying entropic regularization levels. The objective gap (Left) reflects a practical dependence on $\varepsilon$, whereas the parameter error $\|\bar{\mathbf{g}}_t - \mathbf{g}^\star\|^2$ (Right) demonstrates higher robustness to the regularization parameter.}
    \label{fig:eps_dependence}
\end{figure}

\paragraph{Large-Scale Application: Color Transfer.}
To demonstrate scalability, we replicate the color transfer task of \cite{kemertastruncated} on high-resolution images (n=$512^2$ and n=$1024^2$ pixels). We employ our PASGD solver on a single modern GPU using a batch size of $m_b=32$ and a robust, non-tuned learning rate schedule $\eta_t = \frac{n\rho_2}{\e^{-1} + t^{2/3}}$. We set $\varepsilon = 5\cdot 10^{-3}$. Our solver exhibits great efficiency gains compared to the state-of-the-art: while \citet{kemertastruncated} reports a runtime of $\approx 10$ hours for the $1024^2$ task (using \texttt{PyKeops} for memory management), our method converges in just 30 minutes (and $\approx 2$ minutes for $512^2$). This strictly linear $\mathcal{O}(n)$ complexity, both in memory and compute, establishes Entropic UOT via PASGD as a scalable alternative to $\mathcal{O}(n^2)$ per iteration OT baselines for large-scale tasks.

\begin{figure}[ht]
    \centering
    \subfigure[Source]{
        \includegraphics[width=0.18\textwidth]{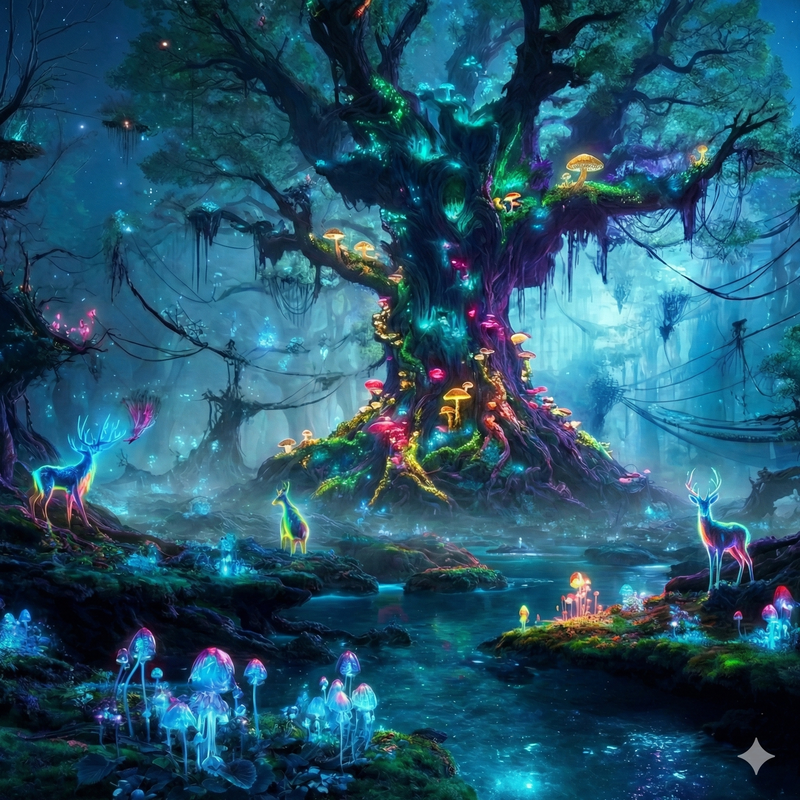}
    }
    \subfigure[Target]{
        \includegraphics[width=0.18\textwidth]{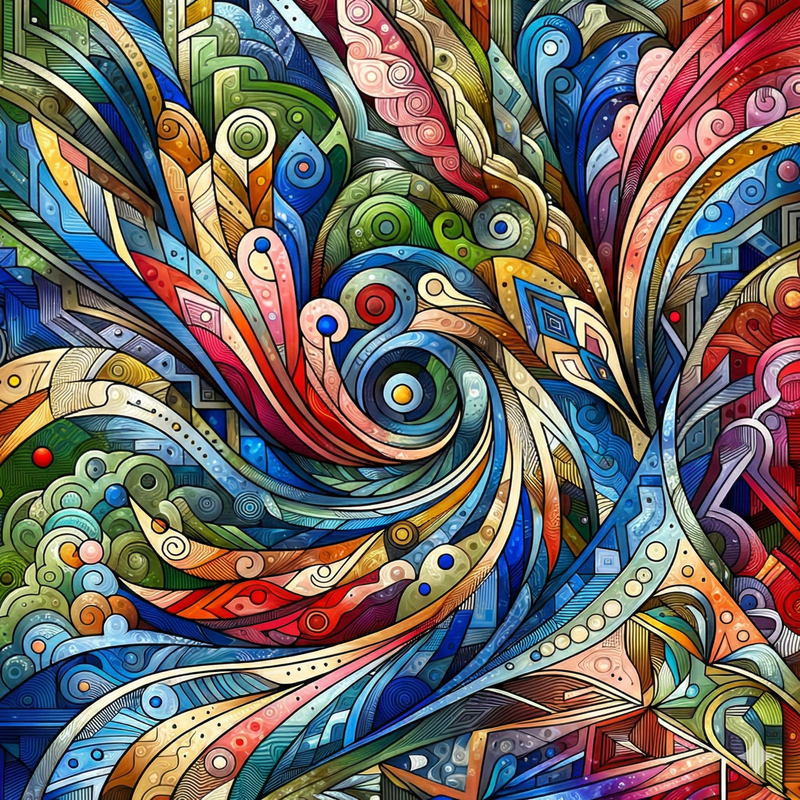}
    }
    \subfigure[$\rho = 0.1$]{
        \includegraphics[width=0.18\textwidth]{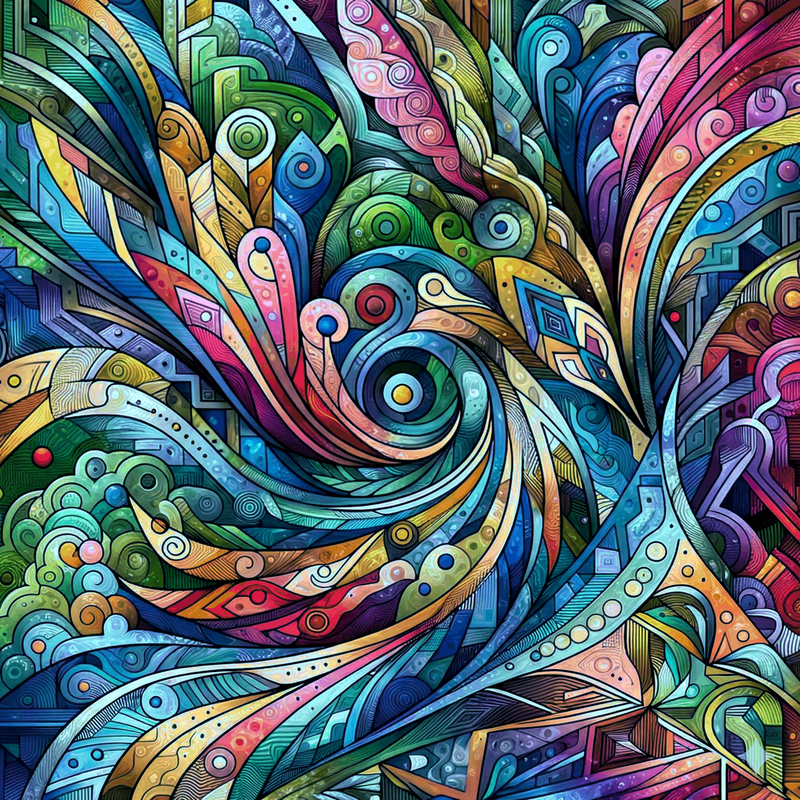}
    }
    \subfigure[$\rho = 1$]{
        \includegraphics[width=0.18\textwidth]{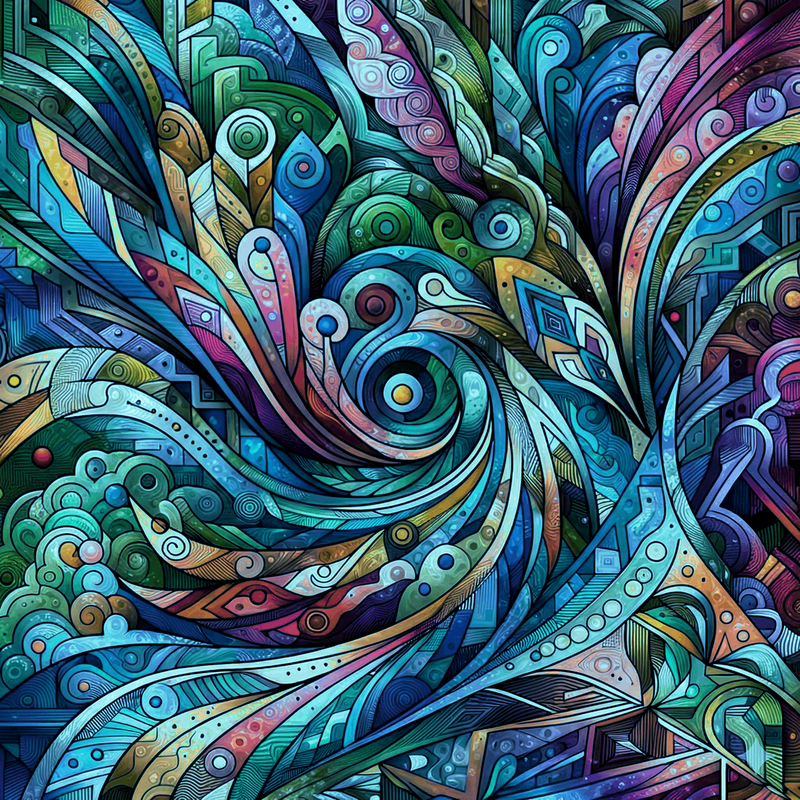}
    }
    \subfigure[$\rho = 10$]{
        \includegraphics[width=0.18\textwidth]{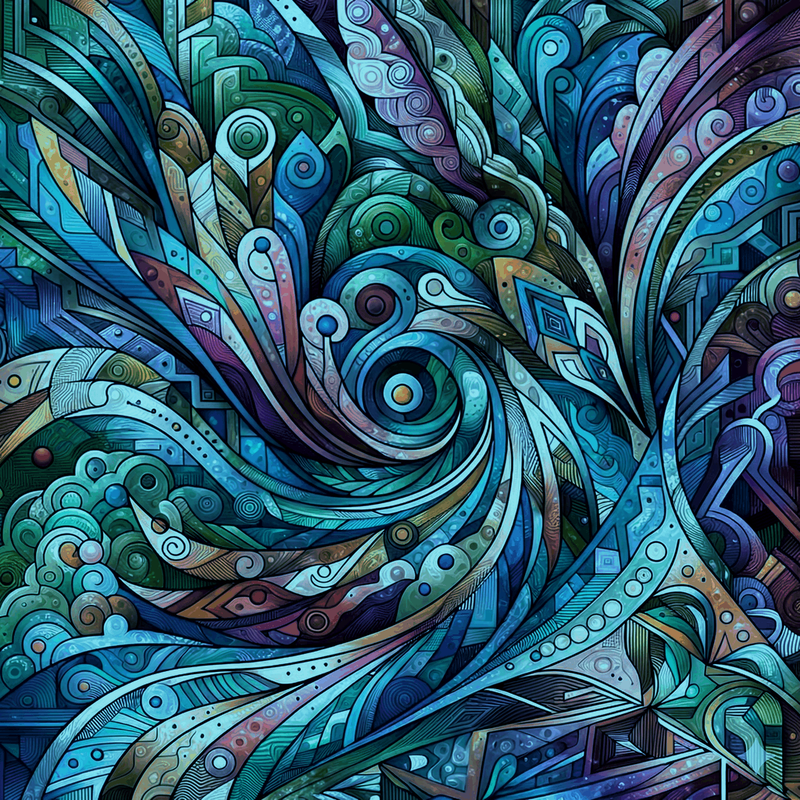}
    }

    \caption{\textbf{High-Resolution Color Transfer ($1024 \times 1024$).} We transport the source color distribution (a) to the target geometry (b). The parameter $\rho$ controls the fidelity of the mass transfer. At $\rho=0.1$ (c), the relaxation allows for partial matching. At $\rho=10$ (e), the penalty enforces nearly balanced transport.}
    \label{fig:five_images}
\end{figure}
\subsection{Discrete - Full Batch setting}\label{subsec:full_discrete}

We now transition to the full-batch setting, which represents the most common scenario for practitioners where measures are discrete or have been pre-discretized.

\begin{setting}[Full Batch Setting]
    We assume $\mu$ and $\nu$ are discrete positive measures supported on $n_1$ and $n_2$ points, respectively:
     $ \mu = \sum_{j=1}^{n_1} \alpha_j \delta_{x_j}, \  \nu = \sum_{j=1}^{n_2} \beta_j \delta_{y_j}.$
     
    To ensure numerical stability, we assume the weights of $\nu$ are balanced relative to the resolution $n_2$: there exist constants $b, B \approx 1$ such that $b/n_2 \le \beta_j \le B/n_2$ for all $j$.
\end{setting}

\paragraph{Data-Dependent Smoothness.}
Standard acceleration schemes rely on a global Lipschitz constant $L$ to determine step sizes. In the UOT semi-dual, however, the global $L$ is prohibitively large, while the local curvature near the optimum is up to $n$ times better. We link here Entropic UOT to the recent line of work on adaptive gradient methods for $(L_0, L_1)$-smooth functions \citep{zhang2019gradient}, where local smoothness scales with the gradient norm. We derive a specialized form of this property that holds explicitly along gradient descent trajectories.

\begin{proposition}[Asymmetric $L_0$--$L_1$ smoothness along a gradient step]
\label{prop:asym_smoothness_line}
Consider the segment
$g(s):=g_t-s\lambda_t \nabla \mathcal J(g_t)$, $s\in[0,1]$, with step size
\[
\lambda_t := \frac{1}{L(g_t)},
\qquad
L(g_t):=
\frac{C}{\varepsilon}\|\nabla \mathcal J_{\trans}(g_t)\|_\infty
+\frac{C\beta_{\max}}{\rho_2}
+\frac{C}{e\varepsilon }\,\|\nabla \mathcal J(g_t)\|_\infty,
\]
where $C = 6e$ for $\chi^2$-source and $C=(2+3\alpha)e$ for KL-source.
Then for any two points $g_1,g_2$ on the segment $\{g(s):s\in[0,1]\}$,
\[
\|\nabla \mathcal J(g_1)-\nabla \mathcal J(g_2)\|
\;\le\;
L(g_t)\,\|g_1-g_2\|.
\]
\end{proposition}

\begin{wrapfigure}{R}{0.5\textwidth}
    \begin{minipage}{0.50\textwidth}
    \vspace{-1em}
        \begin{algorithm}[H]
        \caption{Smoothness Adaptive NAG (with safeguard restarts)}
        \label{alg:adaptive_nag}
        \small
        \setstretch{1.2}
        \begin{algorithmic}[1]
        \STATE \textbf{Input data:} $\mu  = \sum_{i=1}^{n_1} \alpha_i \delta_{x_i}, \nu = \sum_{j=1}^{n_2} \beta_j \delta_{y_j}$
        \STATE \textbf{UOT parameters:} $\varepsilon, \rho_1, \rho_2 > 0$
        \STATE \textbf{Sets:} $\mathcal K := \{g\in\mathbb R^n: g_i \le \rho_2+0.1\},\quad \mathcal K_1 := \{g\in\mathbb R^n: g_i \le \rho_2+1\}$
        \STATE $y_0 \leftarrow g_0 = (0, \dots, 0) \in \mathbb{R}^n$
        \FOR{$t=0$ \textbf{to} $T-1$}
            \STATE $
            \begin{aligned}
            L_t \leftarrow\;
            \tfrac{C}{\varepsilon}\|\nabla \mathcal J_{\trans}(y_t)\|_\infty
            +& \tfrac{C\beta_{\max}}{\rho_2} \\
            &+ \tfrac{C}{e\varepsilon }\,\|\nabla \mathcal J(y_t)\|_\infty
            \end{aligned}
            $
            \STATE $\theta_t \leftarrow \frac{\sqrt{L_t}-\sqrt{\beta_{\min}/\rho_2}}{\sqrt{L_t}+\sqrt{\beta_{\min}/\rho_2}}$
            \STATE $g_{t+1} \leftarrow \Pi_{\mathcal{K_{0.1}}}[y_t - \tfrac{1}{L_t} \nabla\mathcal{J}(y_t)]$
            \STATE $y_{t+1} \leftarrow g_{t+1} + \theta_t(g_{t+1} - g_t)$
             \STATE \textbf{Restart:} \textbf{if} $y_{t+1} \notin \mathcal K_1$ \textbf{then} $y_{t+1} \leftarrow g_t$
        \ENDFOR
        \STATE \textbf{Output:} $g_T$
        \end{algorithmic}
        \end{algorithm}
        \end{minipage}
        \normalsize
\end{wrapfigure}

\noindent This proposition is pivotal. It shows that, rather than relying on a crude worst-case global constant, we can use a local smoothness estimate \(L(g_t)\) that is directly controlled by the computed gradient. As the algorithm converges and \(\nabla \mathcal{J}(g_t)\to 0\), this bound tightens naturally toward \(L(g^\star)\approx \beta_{\max}\bigl(1+\rho_2/\varepsilon\bigr)\), which in turn allows the effective step size to increase. Leveraging this local geometry, we propose an Adaptive Nesterov Accelerated Gradient (ANAG) method. In our setting, each ANAG iteration has complexity \(\mathcal{O}(n_1 n_2)\).

\noindent While ANAG is structurally related to the heuristic adaptive schemes of \citet{malitsky2019adaptive}, which estimate curvature via finite differences, our method uses the analytical upper bound \(L(g_t)\) derived from the problem structure; combined with our \((L_0,L_1)\)-type smoothness control and the projection set, this enables convergence guarantees. Finally, to ensure that all smoothness arguments remain valid, we include a simple safeguard restart whenever the extrapolated point leaves the region \(\mathcal{K}_1\), where \(\mathcal{J}\) has controlled smoothness. This restart is expected to be rare in practice, and often absent, since the optimizer lies in \(\mathcal{K}_0\); it is introduced primarily to simplify the analysis.

\begin{theorem}[Adaptive NAG Convergence Rate]\label{th::anag_cv}(Proof in Appendix \ref{app::proof_th::anag_cv}.) 
    Let $R$ be the number of restarts. Then, the iterates generated by Algorithm \ref{alg:adaptive_nag} satisfy
    \begin{equation}\label{eq:nag_adaptive_rate}
    \mathcal{J}(g_{T+1}) - \mathcal{J}^\star
    \leq
    2^{R}\left( \mathcal{J}(g_0) -\mathcal{J}^\star + \frac{\beta_{\min}}{2\rho_2}\|g_0 - g^\star\|^2\right)
    \prod_{t=0}^{T}\left( 1 - \sqrt{\frac{\beta_{\min}}{\rho_2 L_t}}\right),
    \end{equation}
    
    \noindent Furthermore, the algorithm ensures $y_t \in K_1$ for all $t$, so using $C_{\text{bound}}$ from Lemma \ref{lem:global_grad_bound}, we have $L_t \leq \bar{L} = \mathcal{O}(C_{\text{bound}}/\varepsilon)$ for all $t$. This implies the following rates:
    \begin{enumerate}
        \item \textbf{Global Rate:}  We have at least the contraction rate $1 - \mathcal{O}(\sqrt{\varepsilon/\beta_{\min}\rho_2})$ for both the objective gap and gradient norm.
        \item \textbf{Local Rate:} Since $L_t \leq 2\beta_{\max}\left(\frac{1}{\varepsilon} + \frac{1}{\rho_2} \right) + 2\|\nabla \mathcal{J}(g_{T+1})\|$, assuming $\beta_{\min}/\beta_{\max} \simeq 1$ and substituting the contraction rate of the gradient norm into \eqref{eq:nag_adaptive_rate} shows that, locally, we have a contraction rate of $1 - \mathcal{O}(\sqrt{\varepsilon/\rho_2})$.
    \end{enumerate}
\end{theorem}

Therefore, in our setting, we have a local total complexity of $\mathcal{O}(n_1n_2\sqrt{\rho_2/\varepsilon})\ln(1/\delta_{\text{acc}})$ for $\delta_{\text{acc}}$-accuracy. This result formally confirms our geometric intuition: while the initial convergence may depend on the problem size, the adaptive solver rapidly transitions to a regime where the complexity is governed solely by the local condition number $\mathcal{O}(1/\varepsilon)$, yielding a highly scalable discrete solver.

\paragraph{Experiment: Scale Invariance and Acceleration.}
We validate Theorem \ref{th::anag_cv} using synthetic measures supported on $n$ points in $[0, 1]^{10}$ ($\beta_i = 1/n, \varepsilon = 10^{-2}, \rho_{1,2} = 10$). We compare ANAG against Adaptive GD and two baselines: (i) fixed (conservative) learning rate GD and (ii) fixed learning rate NAG. Figure \ref{fig:nag_scale} highlights the \textbf{scale invariance} of ANAG: trajectories for $n \in \{1500, 3000, 4500\}$ overlap, confirming that the local condition number is independent of problem size $n$. Furthermore, the results isolate the benefits of adaptivity: Adaptive GD significantly outperforms Conservative NAG, demonstrating that exploiting local smoothness ($L(g_t) \ll L_{\text{global}}$) is more critical than blind acceleration. ANAG yields the fastest rates by combining both advantages.

\begin{figure}[h!]
    \centering
    \includegraphics[width=0.5\linewidth]{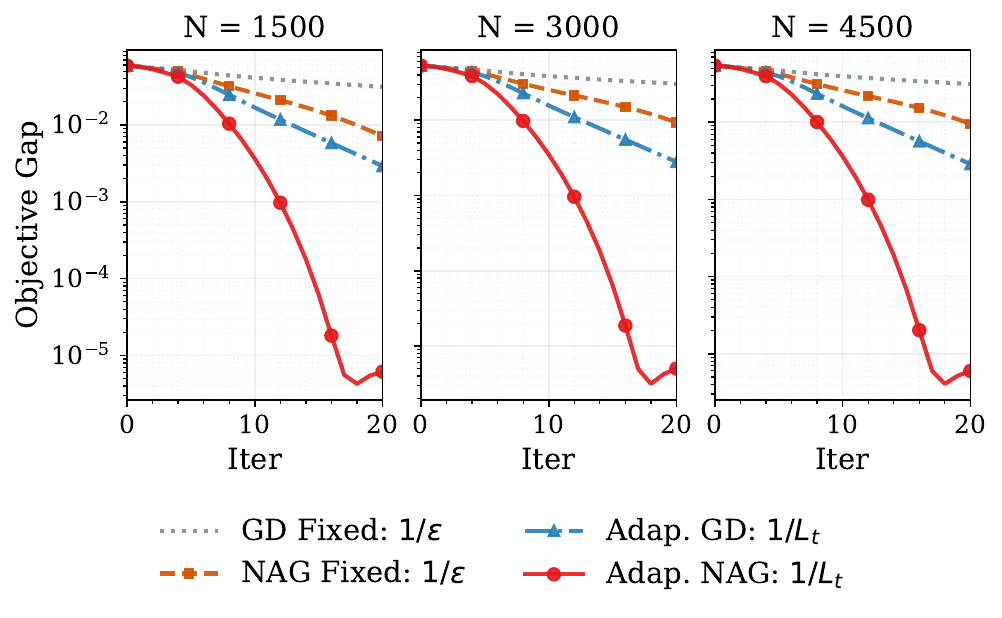}
    \caption{\textbf{Scale Invariance and Adaptive Acceleration.} Convergence on random measures with varying support sizes $n$ ($\varepsilon=0.01, \rho=10$). We compare ANAG against Adaptive GD and Conservative NAG (fixed step $1/L_{\text{global}}$). The overlap of ANAG curves confirms the dimension-independent local complexity, while the superiority of adaptive schemes highlights the benefit of local step sizes.}
    \label{fig:nag_scale}
\end{figure}

\textbf{Remark: ANAG for the KL-KL divergences case.}
Our restriction to the $\chi^2$ target penalty simplifies the analysis by fixing the strong convexity parameter. However, for the $D_1=D_2=\KL$ case, the strong convexity depends on the diagonal of the Hessian, given by the vector $\frac{1}{\rho_2}e^{-\mathbf{g}/\rho_2} \odot \boldsymbol{\beta}$. Assuming the optimal potentials are bounded, one could adapt Algorithm~\ref{alg:adaptive_nag} to update the momentum parameter $\theta_t$ using a data-dependent strong convexity estimate $\mu_t \approx \min_j \left(\frac{\beta_j}{\rho_2}e^{-(g_t)_j/\rho_2} \right)$.

\section*{Comparison to the literature}

\textbf{UOT algorithms.} The primary solvers for Entropic UOT are generalized Sinkhorn algorithms \citep{chizat2018scaling} and their translation-invariant variants \citep{sejourne2022faster}. In the unbalanced setting, Sinkhorn iterations enjoy an enhanced complexity of $\mathcal{O}(n^2/\varepsilon)$ \citep{pham2020unbalanced}, improving upon the $\mathcal{O}(n^2/\varepsilon^2)$ scaling of balanced OT \citep{dvurechensky2018computational}. Beyond Sinkhorn, \citet{nguyen2023unbalanced} proposed a Gradient Extrapolation Method (GEM) for $L_2$-regularized UOT, which produces sparse transport plans. While GEM achieves linear convergence, its condition number $\kappa$ depends on the input measures and scales as $\mathcal{O}(n)$, which is prohibitive for large-scale problems. Alternatively, \citet{chapel2021unbalanced} introduced a Majorization-Minimization scheme with $\mathcal{O}(n^2)$ per-iteration cost, though it currently lacks global convergence rates. In the continuous domain, neural network approximations exist \citep{ gazdieva2024light}, but these methods are computationally heavy and lack convex optimization guarantees.

\begin{table}[h]
    \centering
    \footnotesize 
    \caption{Complexity of UOT algorithms to achieve $\delta$-accuracy, given discrete measures of size $n$. Our methods provide the first accelerated rates for discrete UOT and the first rigorous rates for the semi-discrete regime.}
    \label{tab:algo_comparison}
    
    \begin{tabular}{lccc}
    \toprule
    \textbf{Algorithm} & \textbf{Regime} & \textbf{Regularization} & \textbf{Complexity / Rate} \\
    \midrule
    Gen. Sinkhorn {\scriptsize \citep{pham2020unbalanced}} & Discrete & Entropic & $\mathcal{O}(n^2\ln(1/\delta)/\varepsilon)$ \\
    GEM {\scriptsize \citep{nguyen2023unbalanced}} & Discrete & $L_2$ & $\mathcal{O}(\kappa(n) n^2\ln(1/\delta))$ \\
    MM {\scriptsize \citep{chapel2021unbalanced}} & Discrete & $L_2$, Entropic  & N/A \\
    Neural UOT {\scriptsize \citep{eyring2023unbalancedness}} & Continuous & Entropic & Heuristic \\
    \midrule
    \textbf{Adaptive NAG (Ours)} & \textbf{Discrete} & Entropic & $\tilde{\mathcal{O}}(n^2\ln(1/\delta)/\sqrt{\varepsilon})$ {\footnotesize (Local)} \\
    \textbf{PASGD (Ours)} & \textbf{Semi-Disc.} & Entropic & $\mathcal{O}(n/\varepsilon T)$ \\
    \bottomrule
    \end{tabular}
\end{table}

\textbf{Adaptive Gradient Methods.} The strategy of adapting step sizes via local smoothness estimation was pioneered by \citet{malitsky2019adaptive} and recently extended to $(L_0, L_1)$-smooth functions \citep{gorbunovmethods, vankovoptimizing}. While these general frameworks often require complex algorithmic adjustments to handle potentially unbounded curvature, our analysis exploits the specific geometry of the UOT semi-dual. We show that the smoothness is locally bounded, enabling us to prove the convergence of a simpler Adaptive NAG scheme with tighter complexity guarantees than generic $(L_0, L_1)$ approaches.

\section{Conclusion}
\label{sec:conclusion}

In this work, we demonstrate that the geometry of the Entropic UOT semi-dual is naturally suited for adaptive first-order methods. In the semi-discrete regime, our stochastic scheme provides the first rigorous convergence guarantees for UOT, matching the $\mathcal{O}(n^2/\varepsilon)$ efficiency of Sinkhorn while enabling strictly linear scalability. In the full-batch discrete setting, our Adaptive NAG achieves a superior local accelerated complexity of $\tilde{\mathcal{O}}(n^2/\sqrt{\varepsilon})$ and a global worst-case of $\tilde{\mathcal{O}}(n^{2.5}/\sqrt{\varepsilon})$.

Looking forward, \textbf{exploiting the key geometric properties derived in this work suggests} further potential in bridging UOT with generalized $(L_0, L_1)$-smoothness frameworks to design tailored solvers. Furthermore, promising directions to improve the global worst-case complexity include (i) AdaGrad-type algorithms \citep{duchi2011adaptive} in the semi-discrete setting to further leverage global conditioning and (ii) the integration of decreasing regularization schedules \citep{schmitzer2019stabilized} in the discrete setting. We hypothesize that leveraging the strong convexity of the semi-dual in this manner could eliminate the current initialization overhead, potentially securing a global $\tilde{\mathcal{O}}(n^2/\sqrt{\varepsilon})$ convergence rate.

\bibliography{references}


\doparttoc 
\faketableofcontents 

\part{} 

\appendix
\addcontentsline{toc}{section}{Appendix} 
\parttoc 

\section{Nomenclature}
\label{app:nomenclature}

Table~\ref{tab:notation} summarizes the specific mathematical notations and operators used throughout the paper.

\begin{table}[h]
    \centering
    \renewcommand{\arraystretch}{1.3}
    \caption{Nomenclature and List of Symbols}
    \label{tab:notation}
    \small
    \begin{tabular}{@{}p{0.22\textwidth} p{0.73\textwidth}@{}}
        \toprule
        \textbf{Symbol} & \textbf{Description} \\
        \midrule
        \multicolumn{2}{c}{\textsc{Measure Theory \& Operators}} \\
        \midrule
        $\mathcal{M}_+(\mathcal{X})$ & Space of finite non-negative measures on $\mathcal{X}$. \\
        $\alpha \ll \beta$ & Absolute continuity: measure $\alpha$ is dominated by $\beta$. \\
        $\delta_x$ & Dirac mass at location $x$. \\
        $f \lesssim f$ & Inequality up to a positive constant: $f (\cdot) \le C \cdot g(\cdot)$ for some universal $C > 0$. \\
        $\langle \cdot, \cdot \rangle$ & Standard Euclidean inner product. \\
        $\odot$ & Hadamard (element-wise) product. \\
        $\Pi_{\mathcal{K}}$ & Euclidean projection onto the set $\mathcal{K}$. \\
        \midrule
        \multicolumn{2}{c}{\textsc{Unbalanced OT \& Divergences}} \\
        \midrule
        $D_\varphi(\cdot \mid \cdot)$ & Csisz\'ar $\varphi$-divergence defined by convex generator $\varphi$. \\
        $\varphi^c$ & Convex conjugate (Legendre-Fenchel transform) of $\varphi$. \\
        $\rho_1, \rho_2$ & Marginal penalty weights for source and target, respectively. \\
        $\alpha$ & Scaling exponent for the KL-source case: $\alpha = \frac{\varepsilon}{\varepsilon + \rho_1}$. \\
        $W(\cdot)$ & Lambert $W$ function (inverse of $z \mapsto z e^z$). \\
        \midrule
        \multicolumn{2}{c}{\textsc{Semi-Dual Geometry}} \\
        \midrule
        $B_j(x, \mathbf{g})$ & Gibbs kernel term: $\beta_j \exp((g_j - c(x, y_j))/\varepsilon)$. \\
        $Z(x, \mathbf{g})$ & Partition function: $\sum_{j=1}^n B_j(x, \mathbf{g})$. \\
        $w_k(x, \mathbf{g})$ & Normalized transport weights: $B_k(x, \mathbf{g}) / Z(x, \mathbf{g})$. \\
        $U(x, \mathbf{g})$ & Scaled partition function ($\chi^2$-source): $\frac{\rho_1}{\varepsilon}e^{\rho_1/\varepsilon}Z(x, \mathbf{g})$. \\
        $\mathcal{J}_{\text{trans}}(\mathbf{g})$ & Transport component of the semi-dual functional. \\
        \midrule
        \multicolumn{2}{c}{\textsc{Optimization \& Algorithms}} \\
        \midrule
        $\mathcal{K}_\delta$ & Effective domain (feasible set): $\{ \mathbf{g} \in \mathbb{R}^n \mid g_j \le \rho_2 + \delta \}$. \\
        $L(\mathbf{g})$ & Local smoothness upper bound at $\mathbf{g}$. \\
        $\widehat{\nabla}\mathcal{J}$ & Stochastic gradient estimator (batch size $m_b$). \\
        \bottomrule
    \end{tabular}
\end{table}

\newpage
\section{Derivation of the semi-dual functionals. }\label{appendix:semi_dual}

\textbf{Summary of the semi-dual objectives, gradients, and recovery of the dual potentials. } 
\label{prop:semi_dual_recap}
Assume $\nu=\sum_{j=1}^n \beta_j\delta_{y_j}$ with $\beta_j>0$ and fix the target marginal penalty
$D_2=D_{\chi^2}$ with weight $\rho_2>0$. Define, for $x\in\X$ and $g\in\R^n$,
\[
B_j(x,g):=\beta_j\exp\!\Big(\tfrac{g_j-c(x,y_j)}{\varepsilon}\Big),\qquad
Z(x,g):=\sum_{j=1}^n B_j(x,g),\qquad
w_j(x,g):=\frac{B_j(x,g)}{Z(x,g)}.
\]
Then any dual maximizer satisfies $g^\star_j\le \rho_2$ for all $j$, and the dual problem reduces to the minimization
of a strictly convex semi-dual functional $\J(g)$ of the form
\[
\J(g)=\J_{\trans}(g)\;+\;\sum_{j=1}^n \beta_j\Big(\tfrac{g_j^2}{2\rho_2}-g_j\Big),
\]
where the transport term $\J_{\trans}$ depends on the choice of the source divergence $D_1$:

\noindent \textbf{KL-source ($D_1=\KL$).} Let $\alpha:=\frac{\varepsilon}{\rho_1+\varepsilon}$. Then
\begin{equation}
\label{eq:recap_J_KL}
\J_{\trans}(g)=(\rho_1+\varepsilon)\int_{\X} Z(x,g)^\alpha\,\d\mu(x),
\end{equation}
and the gradient is
\begin{equation}
\label{eq:recap_grad_KL}
\nabla_k \J(g)=\int_{\X} Z(x,g)^\alpha\,w_k(x,g)\,\d\mu(x)\;+\;\beta_k\Big(\tfrac{g_k}{\rho_2}-1\Big).
\end{equation}
Moreover, for any fixed $g$, the maximizer in the eliminated variable $f$ is unique and given in closed form by
\begin{equation}
\label{eq:recap_fstar_KL}
f^\star(x;g)= -\frac{\rho_1\varepsilon}{\rho_1+\varepsilon}\ln Z(x,g),
\qquad x\in\X.
\end{equation}

\noindent \textbf{$\chi^2$-source ($D_1=D_{\chi^2}$).} Define
\[
U(x,g):=\frac{\rho_1}{\varepsilon}e^{\rho_1/\varepsilon}Z(x,g),
\qquad W(\cdot)\ \text{the Lambert function}.
\]
Then
\begin{equation}
\label{eq:recap_J_chi2}
\J_{\trans}(g)=\frac{\varepsilon^2}{\rho_1}\int_{\X}\Big(W(U(x,g))+\tfrac12 W(U(x,g))^2\Big)\,\d\mu(x),
\end{equation}
and the gradient is
\begin{equation}
\label{eq:recap_grad_chi2}
\nabla_k \J(g)=\int_{\X}\frac{\varepsilon}{\rho_1}W(U(x,g))\,w_k(x,g)\,\d\mu(x)\;+\;\beta_k\Big(\tfrac{g_k}{\rho_2}-1\Big).
\end{equation}
Moreover, for any fixed $g$, the maximizer in the eliminated variable $f$ is unique and given by
\begin{equation}
\label{eq:recap_fstar_chi2}
f^\star(x;g)=\rho_1-\varepsilon\,W(U(x,g)),
\qquad x\in\X.
\end{equation}

Finally, if $g^\star$ minimizes $\J$, then $(f^\star(\cdot;g^\star),g^\star)$ is a maximizer of the original dual
\eqref{eq:dual_start_app}. The associated optimal coupling is recovered from the dual potentials by
\[
\frac{\d\pi^\star}{\d(\mu\otimes\nu)}(x,y_j)=
\exp\!\left(\frac{f^\star(x;g^\star)+g^\star_j-c(x,y_j)}{\varepsilon}\right).
\]

\begin{proof}
We consider entropic unbalanced OT with discrete target
$\nu=\sum_{j=1}^n \beta_j \delta_{y_j}$, $\beta_j>0$.
Starting from the dual (up to additive constants independent of $(f,\mathbf g)$),
\begin{align}
\label{eq:dual_start_app}
\sup_{f,\mathbf g}\Bigg\{&
-\varepsilon \int_{\mathcal X}\sum_{j=1}^n \beta_j
\exp\!\Big(\tfrac{f(x)+g_j-c(x,y_j)}{\varepsilon}\Big)\d\mu(x)
-\rho_1 \int_{\mathcal X}\varphi_1^c\!\Big(-\tfrac{f(x)}{\rho_1}\Big)\d\mu(x)
-\rho_2 \sum_{j=1}^n \beta_j \,\varphi_{\chi^2}^c\!\Big(-\tfrac{g_j}{\rho_2}\Big)
\Bigg\}.
\end{align}

\paragraph{Pearson $\chi^2$ generator and conjugate.}
The Pearson $\chi^2$ divergence between measures $\pi \ll \nu$ is generated by
\[
\varphi_{\chi^2}(t)=\tfrac12 (t-1)^2,\qquad t\ge 0,
\]
so that $D_{\chi^2}(\pi\|\nu)=\int \varphi_{\chi^2}\!\big(\tfrac{d\pi}{d\nu}\big)\d\nu$.
Its convex conjugate over $t\ge 0$ is
\begin{equation}
\label{eq:chi2_conj_app}
\varphi_{\chi^2}^c(s)=\sup_{t\ge 0}\{st-\tfrac12(t-1)^2\}
=
\begin{cases}
s+\tfrac12 s^2, & s\ge -1,\\
-\tfrac12, & s<-1.
\end{cases}
\end{equation}

\paragraph{Softmax quantities.}
For any $x\in\mathcal X$ and $\mathbf g\in\mathbb R^n$, define
\begin{equation}
\label{eq:defs_BZw_app}
B_j(x,\mathbf g):=\beta_j\exp\!\Big(\tfrac{g_j-c(x,y_j)}{\varepsilon}\Big),\qquad
Z(x,\mathbf g):=\sum_{j=1}^n B_j(x,\mathbf g),\qquad
w_j(x,\mathbf g):=\frac{B_j(x,\mathbf g)}{Z(x,\mathbf g)}.
\end{equation}
Then
\[
\sum_{j=1}^n \beta_j \exp\!\Big(\tfrac{f(x)+g_j-c(x,y_j)}{\varepsilon}\Big)
=
e^{f(x)/\varepsilon} Z(x,\mathbf g).
\]

Let $s_j=-g_j/\rho_2$. Using \eqref{eq:chi2_conj_app},
\begin{equation}
\label{eq:target_term_piecewise_app}
-\rho_2\,\varphi_{\chi^2}^c(s_j)
=
\begin{cases}
g_j-\tfrac{g_j^2}{2\rho_2}, & g_j\le \rho_2 \ (s_j\ge -1),\\[2pt]
\tfrac{\rho_2}{2}, & g_j> \rho_2 \ (s_j<-1).
\end{cases}
\end{equation}
In our dual objective, the only other dependence on $g_j$ is through the entropic term
$-\varepsilon \int e^{f/\varepsilon} B_j(x,\mathbf g)\d\mu(x)$, which is strictly decreasing in $g_j$.
Therefore, increasing $g_j$ above $\rho_2$ can only decrease the objective, hence at any maximizer one has
\begin{equation}
\label{eq:gj_le_rho2_app}
g_j^\star\le \rho_2\quad\text{for all }j.
\end{equation}
As a consequence, we may use the quadratic branch in \eqref{eq:target_term_piecewise_app} and write
\begin{equation}
\label{eq:target_term_quad_app}
-\rho_2 \sum_{j=1}^n \beta_j\,\varphi_{\chi^2}^c\!\Big(-\tfrac{g_j}{\rho_2}\Big)
=
\sum_{j=1}^n \beta_j\Big(g_j-\tfrac{g_j^2}{2\rho_2}\Big).
\end{equation}

Plugging \eqref{eq:defs_BZw_app} and \eqref{eq:target_term_quad_app} into \eqref{eq:dual_start_app} yields
\begin{align}
\label{eq:dual_after_discrete_app}
\sup_{f,\mathbf g}\Bigg\{&
\int_{\mathcal X}\Big[
-\varepsilon\,e^{f(x)/\varepsilon}Z(x,\mathbf g)
-\rho_1\,\varphi_1^c\!\Big(-\tfrac{f(x)}{\rho_1}\Big)
\Big]\d\mu(x)
+\sum_{j=1}^n \beta_j\Big(g_j-\tfrac{g_j^2}{2\rho_2}\Big)
\Bigg\}.
\end{align}
For fixed $\mathbf g$, the maximization over $f$ is separable in $x$.

\paragraph{Case A: KL-source ($\varphi_1=\KL$).}
Here $\varphi_{\KL}^c(s)=e^s-1$, hence
\[
-\rho_1\,\varphi_{\KL}^c\!\Big(-\tfrac{f}{\rho_1}\Big)
=
-\rho_1\big(e^{-f/\rho_1}-1\big)
=
-\rho_1 e^{-f/\rho_1}+\rho_1.
\]
Fix $x$ and abbreviate $Z=Z(x,\mathbf g)$. We maximize over $f\in\mathbb R$:
\begin{equation}
\label{eq:phi_KL_pointwise_app}
\Phi_x(f)
:=
-\varepsilon Z e^{f/\varepsilon}-\rho_1 e^{-f/\rho_1}+\rho_1.
\end{equation}
The first-order condition gives
\[
0=\Phi_x'(f)=-Z e^{f/\varepsilon}+e^{-f/\rho_1}
\quad\Longleftrightarrow\quad
e^{-f/\rho_1}=Z e^{f/\varepsilon}.
\]
Let $\alpha:=\frac{\varepsilon}{\rho_1+\varepsilon}$. The unique maximizer is
\begin{equation}
\label{eq:fstar_KL_app}
f^\star(x;\mathbf g)= -\frac{\rho_1\varepsilon}{\rho_1+\varepsilon}\ln Z(x,\mathbf g)
\quad\text{and}\quad
e^{-f^\star/\rho_1}=Z^\alpha,\;\; e^{f^\star/\varepsilon}=Z^{\alpha-1}.
\end{equation}
Plugging into \eqref{eq:phi_KL_pointwise_app} yields the exact pointwise optimum value
\begin{equation}
\label{eq:val_KL_app}
\sup_f \Phi_x(f)
=
\rho_1-(\rho_1+\varepsilon)\,Z(x,\mathbf g)^\alpha.
\end{equation}
Therefore the dual reduces to
\begin{align}
\label{eq:dual_reduced_KL_app}
\sup_{\mathbf g}\Bigg\{&
\rho_1\,\mu(\mathcal X)
-(\rho_1+\varepsilon)\int_{\mathcal X} Z(x,\mathbf g)^\alpha\d\mu(x)
+\sum_{j=1}^n \beta_j\Big(g_j-\tfrac{g_j^2}{2\rho_2}\Big)
\Bigg\}.
\end{align}
Equivalently, maximizing \eqref{eq:dual_reduced_KL_app} is the same as minimizing the strictly convex objective
\begin{equation}
\label{eq:J_KL_app}
\mathcal J_{\KL}(\mathbf g)
:=
(\rho_1+\varepsilon)\int_{\mathcal X} Z(x,\mathbf g)^\alpha\d\mu(x)
+\sum_{j=1}^n \beta_j\Big(\tfrac{g_j^2}{2\rho_2}-g_j\Big),
\end{equation}
and the optimal dual value equals $\rho_1\mu(\mathcal X)-\inf_{\mathbf g}\mathcal J_{\KL}(\mathbf g)$.

\paragraph{Case B: $\chi^2$-source ($\varphi_1=\chi^2$).}
Using \eqref{eq:chi2_conj_app} with $s=-f/\rho_1$, the quadratic branch (valid when $f\le\rho_1$) gives
\[
-\rho_1\,\varphi_{\chi^2}^c\!\Big(-\tfrac{f}{\rho_1}\Big)
=
-\rho_1\Big(-\tfrac{f}{\rho_1}+\tfrac12\tfrac{f^2}{\rho_1^2}\Big)
=
f-\tfrac{f^2}{2\rho_1}.
\]
If $f>\rho_1$, the penalty becomes constant $-\rho_1(-\tfrac12)=\rho_1/2$ while the entropic term
$-\varepsilon Z e^{f/\varepsilon}$ strictly decreases with $f$, so at any maximizer we have $f^\star(x)\le\rho_1$.
Fix $x$ and write $Z=Z(x,\mathbf g)$. We maximize over $f\le\rho_1$:
\begin{equation}
\label{eq:phi_chi2_pointwise_app}
\Phi_x(f)
:=
-\varepsilon Z e^{f/\varepsilon}+f-\tfrac{f^2}{2\rho_1}.
\end{equation}
The first-order condition is
\begin{equation}
\label{eq:FOC_chi2_app}
0=\Phi_x'(f)=-Z e^{f/\varepsilon}+1-\tfrac{f}{\rho_1}
\quad\Longleftrightarrow\quad
1-\tfrac{f}{\rho_1}=Z e^{f/\varepsilon}.
\end{equation}
Let $u:=1-\tfrac{f}{\rho_1}\ge 0$. Then \eqref{eq:FOC_chi2_app} becomes
\[
u=Z\exp\!\Big(\tfrac{\rho_1}{\varepsilon}(1-u)\Big)
=Z e^{\rho_1/\varepsilon}\,e^{-(\rho_1/\varepsilon)u}.
\]
Equivalently,
\[
\Big(\tfrac{\rho_1}{\varepsilon}u\Big)\exp\!\Big(\tfrac{\rho_1}{\varepsilon}u\Big)
=
\frac{\rho_1}{\varepsilon}Z e^{\rho_1/\varepsilon}.
\]
Define
\begin{equation}
\label{eq:U_def_app}
U(x,\mathbf g):=\frac{\rho_1}{\varepsilon}Z(x,\mathbf g)e^{\rho_1/\varepsilon},
\qquad
W(\cdot)\ \text{the Lambert function}.
\end{equation}
Then the unique maximizer is
\begin{equation}
\label{eq:fstar_chi2_app}
u(x,\mathbf g)=\frac{\varepsilon}{\rho_1}W\!\big(U(x,\mathbf g)\big),
\qquad
f^\star(x;\mathbf g)=\rho_1-\varepsilon\,W\!\big(U(x,\mathbf g)\big)\le \rho_1.
\end{equation}
To evaluate the optimum value, note from \eqref{eq:FOC_chi2_app} that $Z e^{f^\star/\varepsilon}=u$.
Plugging $f^\star=\rho_1(1-u)$ into \eqref{eq:phi_chi2_pointwise_app} gives
\[
\sup_f \Phi_x(f)
=
-\varepsilon u+\rho_1(1-u)-\tfrac{\rho_1}{2}(1-u)^2
=
\tfrac{\rho_1}{2}-\varepsilon u-\tfrac{\rho_1}{2}u^2.
\]
Using \eqref{eq:fstar_chi2_app} with $u=\frac{\varepsilon}{\rho_1}w$ and $w:=W(U)$, we obtain the exact value
\begin{equation}
\label{eq:val_chi2_app}
\sup_f \Phi_x(f)
=
\frac{\rho_1}{2}
-\frac{\varepsilon^2}{\rho_1}\Big(w+\tfrac12 w^2\Big),
\qquad w=W\!\big(U(x,\mathbf g)\big).
\end{equation}
Therefore the dual reduces to
\begin{align}
\label{eq:dual_reduced_chi2_app}
\sup_{\mathbf g}\Bigg\{&
\frac{\rho_1}{2}\,\mu(\mathcal X)
-\frac{\varepsilon^2}{\rho_1}\int_{\mathcal X}\Big(W(U)+\tfrac12 W(U)^2\Big)\d\mu
+\sum_{j=1}^n \beta_j\Big(g_j-\tfrac{g_j^2}{2\rho_2}\Big)
\Bigg\}.
\end{align}
Equivalently, maximizing \eqref{eq:dual_reduced_chi2_app} is the same as minimizing the strictly convex objective
\begin{equation}
\label{eq:J_chi2_app}
\mathcal J_{\chi^2}(\mathbf g)
:=
\frac{\varepsilon^2}{\rho_1}\int_{\mathcal X}\Big(W(U(x,\mathbf g))+\tfrac12 W(U(x,\mathbf g))^2\Big)\d\mu(x)
+\sum_{j=1}^n \beta_j\Big(\tfrac{g_j^2}{2\rho_2}-g_j\Big),
\end{equation}
and the optimal dual value equals $\frac{\rho_1}{2}\mu(\mathcal X)-\inf_{\mathbf g}\mathcal J_{\chi^2}(\mathbf g)$.
\end{proof}
\subsection*{Gradients of the semi-dual objectives}
\label{app:gradients_semidual}

We now differentiate $\mathcal J(\mathbf g)$ in the two cases.
First note that, from \eqref{eq:defs_BZw_app},
\begin{equation}
\label{eq:dZ_dgk_app}
\frac{\partial Z(x,\mathbf g)}{\partial g_k}
=\frac{1}{\varepsilon}B_k(x,\mathbf g)
=\frac{1}{\varepsilon}Z(x,\mathbf g)\,w_k(x,\mathbf g).
\end{equation}
Also, the target quadratic term always contributes
\begin{equation}
\label{eq:target_grad_app}
\frac{\partial}{\partial g_k}\sum_{j=1}^n \beta_j\Big(\tfrac{g_j^2}{2\rho_2}-g_j\Big)
=\beta_k\Big(\tfrac{g_k}{\rho_2}-1\Big).
\end{equation}

\paragraph{Gradient, KL-source.}
From \eqref{eq:J_KL_app}, using \eqref{eq:dZ_dgk_app},
\[
\frac{\partial}{\partial g_k} Z(x,\mathbf g)^\alpha
=\alpha Z^{\alpha-1}\frac{\partial Z}{\partial g_k}
=\frac{\alpha}{\varepsilon}Z^\alpha w_k
=\frac{1}{\rho_1+\varepsilon}Z^\alpha w_k,
\]
hence
\begin{equation}
\label{eq:grad_KL_app}
\nabla_k \mathcal J_{\KL}(\mathbf g)
=
\int_{\mathcal X} Z(x,\mathbf g)^\alpha\,w_k(x,\mathbf g)\d\mu(x)
+\beta_k\Big(\tfrac{g_k}{\rho_2}-1\Big).
\end{equation}

\paragraph{Gradient, $\chi^2$-source.}
Let $w(x,\mathbf g):=W(U(x,\mathbf g))$ with $U$ defined in \eqref{eq:U_def_app}.
Since $U(x,\mathbf g)=\frac{\rho_1}{\varepsilon}e^{\rho_1/\varepsilon}Z(x,\mathbf g)$, we have
\[
\frac{\partial w}{\partial g_k}
= W'(U)\frac{\partial U}{\partial g_k}
= \frac{w}{U(1+w)}\cdot \frac{U}{Z}\cdot \frac{\partial Z}{\partial g_k}
= \frac{w}{Z(1+w)}\cdot \frac{1}{\varepsilon}B_k.
\]
Differentiate the integrand in \eqref{eq:J_chi2_app}:
\[
\frac{\partial}{\partial g_k}\Big(w+\tfrac12 w^2\Big)
=(1+w)\frac{\partial w}{\partial g_k}
=\frac{w}{Z}\cdot \frac{1}{\varepsilon}B_k
=\frac{w}{\varepsilon}\,w_k.
\]
Therefore
\begin{equation}
\label{eq:grad_chi2_app}
\nabla_k \mathcal J_{\chi^2}(\mathbf g)
=
\int_{\mathcal X}\frac{\varepsilon}{\rho_1}W\!\big(U(x,\mathbf g)\big)
\,w_k(x,\mathbf g)\d\mu(x)
+\beta_k\Big(\tfrac{g_k}{\rho_2}-1\Big).
\end{equation}

\paragraph{Summary in density form.}
For KL-source, define $\alpha=\frac{\varepsilon}{\rho_1+\varepsilon}$ and
\[
\sigma_{\KL}(x,\mathbf g):=Z(x,\mathbf g)^\alpha.
\]
For $\chi^2$-source, define $U$ as in \eqref{eq:U_def_app} and
\[
\sigma_{\chi^2}(x,\mathbf g):=\frac{\varepsilon}{\rho_1}W(U(x,\mathbf g)).
\]
Then \eqref{eq:grad_KL_app} and \eqref{eq:grad_chi2_app} can be written uniformly as
\[
\nabla_k \mathcal J(\mathbf g)
=
\int_{\mathcal X}\sigma(x,\mathbf g)\,w_k(x,\mathbf g)\d\mu(x)
+\frac{\beta_k}{\rho_2}g_k-\beta_k\ . 
\]

\section{Proof of Theorem \ref{thm:smoothness_unified}: Smoothness Bound via Gradient Transport}\label{app::proof_thm_smoothness_unified}

\textbf{Theorem \ref{thm:smoothness_unified} Smoothness Bound via Gradient Transport}.
For all $\mathbf{g} \in \mathbb{R}^n$, the operator norm of the Hessian satisfies:
\begin{equation}
    \|\nabla^2 \mathcal{J}(\mathbf{g})\|_{\mathrm{op}} \le \frac{1}{\varepsilon} \|\nabla \mathcal{J}_{\mathrm{trans}}(\mathbf{g})\|_\infty + \frac{\beta_{\max}}{\rho_2}.
\end{equation}

\begin{proof}
We derive the Hessian in both cases. A common ingredient is the derivative of the softmax weights
$w_k(x,\mathbf g)=B_k(x,\mathbf g)/Z(x,\mathbf g)$.
Using $\partial_{g_l}B_k=\frac{1}{\varepsilon}\delta_{kl}B_k$ and the quotient rule,
\begin{equation}
\label{eq:weight_deriv}
\frac{\partial w_k}{\partial g_l}
=\frac{1}{\varepsilon}w_k(\delta_{kl}-w_l).
\end{equation}

\paragraph{Hessian Derivation: KL-Source Case.}
In the KL-source case, the transport density is $\sigma(x,\mathbf g)=Z(x,\mathbf g)^\alpha$ with
$\alpha=\frac{\varepsilon}{\rho_1+\varepsilon}$.
First,
\[
\frac{\partial \sigma}{\partial g_l}
=\alpha Z^{\alpha-1}\frac{\partial Z}{\partial g_l}
=\alpha Z^{\alpha-1}\cdot \frac{1}{\varepsilon}Z w_l
=\frac{\alpha}{\varepsilon}\sigma w_l.
\]
Then, applying the product rule and \eqref{eq:weight_deriv},
\begin{align*}
\frac{\partial}{\partial g_l}(\sigma w_k)
&=\Big(\frac{\partial\sigma}{\partial g_l}\Big)w_k+\sigma\Big(\frac{\partial w_k}{\partial g_l}\Big) \\
&=\frac{\alpha}{\varepsilon}\sigma w_l w_k + \frac{\sigma}{\varepsilon}w_k(\delta_{kl}-w_l) \\
&=\frac{\sigma}{\varepsilon}\Big[w_k\delta_{kl}-(1-\alpha)w_k w_l\Big],
\end{align*}
where $1-\alpha>0$.

\paragraph{Hessian Derivation: $\chi^2$-Source Case.}
For the semi-dual with $\chi^2$-source obtained by eliminating $f$, the gradient of the transport
term takes the form
\[
\nabla_k \mathcal J_{\mathrm{trans}}(\mathbf g)
=\int_{\mathcal X}\sigma(x,\mathbf g)\,w_k(x,\mathbf g)\d\mu(x),
\qquad
\sigma(x,\mathbf g)=\frac{\varepsilon}{\rho_1}W(U(x,\mathbf g)),
\]
with
\[
U(x,\mathbf g)=\frac{\rho_1}{\varepsilon}Z(x,\mathbf g)e^{\rho_1/\varepsilon}.
\]
We compute $\partial_{g_l}\sigma$. Using $W'(z)=\frac{W(z)}{z(1+W(z))}$ and
$\partial_{g_l}U=\frac{1}{\varepsilon}Uw_l$ (since $U\propto Z$ and $\partial_{g_l}Z=\frac{1}{\varepsilon}Zw_l$),
\begin{align*}
\frac{\partial\sigma}{\partial g_l}
&=\frac{\varepsilon}{\rho_1}W'(U)\,\frac{\partial U}{\partial g_l}
=\frac{\varepsilon}{\rho_1}\cdot \frac{W(U)}{U(1+W(U))}\cdot \frac{1}{\varepsilon}U w_l \\
&=\frac{1}{\rho_1}\frac{W(U)}{1+W(U)}\,w_l
=\frac{\sigma}{\varepsilon(1+W(U))}\,w_l,
\end{align*}
where in the last equality we used $\sigma=\frac{\varepsilon}{\rho_1}W(U)$.

Applying the product rule and \eqref{eq:weight_deriv},
\begin{align*}
\frac{\partial}{\partial g_l}(\sigma w_k)
&=\Big(\frac{\partial\sigma}{\partial g_l}\Big)w_k+\sigma\Big(\frac{\partial w_k}{\partial g_l}\Big) \\
&=\frac{\sigma}{\varepsilon(1+W)}w_l w_k+\frac{\sigma}{\varepsilon}w_k(\delta_{kl}-w_l) \\
&=\frac{\sigma}{\varepsilon}\Big[w_k\delta_{kl}-\Big(1-\frac{1}{1+W(U)}\Big)w_k w_l\Big] \\
&=\frac{\sigma}{\varepsilon}\Big[w_k\delta_{kl}-c(x)\,w_k w_l\Big],
\end{align*}
with the nonnegative coefficient
\[
c(x):=\frac{W(U(x,\mathbf g))}{1+W(U(x,\mathbf g))}\in[0,1).
\]

\paragraph{Unified Spectral Bound.}
In both cases, the Hessian of the full objective $\mathcal J(\mathbf g)$
(transport term plus the target quadratic penalty $\sum_j \beta_j(\frac{g_j^2}{2\rho_2}-g_j)$)
can be written as
\[
\nabla^2 \mathcal J(\mathbf g)
=\int_{\mathcal X}\frac{\sigma(x,\mathbf g)}{\varepsilon}
\Big(\mathrm{diag}(\mathbf w)-c(x)\,\mathbf w\mathbf w^\top\Big)\d\mu(x)
+\frac{1}{\rho_2}\mathrm{diag}(\boldsymbol\beta),
\]
where $c(x)=1-\alpha$ in the KL-source case and $c(x)=\frac{W(U)}{1+W(U)}$ in the $\chi^2$-source case,
hence always $c(x)\ge 0$.

Since $\mathbf w\mathbf w^\top\succeq 0$, subtracting $c(x)\mathbf w\mathbf w^\top$ decreases eigenvalues:
\[
\mathrm{diag}(\mathbf w)-c(x)\mathbf w\mathbf w^\top \preceq \mathrm{diag}(\mathbf w).
\]
Therefore
\begin{align*}
\nabla^2 \mathcal J(\mathbf g)
&\preceq \int_{\mathcal X}\frac{\sigma(x,\mathbf g)}{\varepsilon}\mathrm{diag}(\mathbf w)\d\mu(x)
+\frac{1}{\rho_2}\mathrm{diag}(\boldsymbol\beta) \\
&=\mathrm{diag}\!\left(\frac{1}{\varepsilon}\int_{\mathcal X}\sigma(x,\mathbf g)\,\mathbf w(x,\mathbf g)\d\mu(x)
+\frac{\boldsymbol\beta}{\rho_2}\right).
\end{align*}
Recognizing $\int \sigma w_k\d\mu=[\nabla \mathcal J_{\mathrm{trans}}(\mathbf g)]_k$, we obtain
\[
\|\nabla^2\mathcal J(\mathbf g)\|_{\mathrm{op}}
\le \max_k\left(\frac{1}{\varepsilon}[\nabla \mathcal J_{\mathrm{trans}}(\mathbf g)]_k+\frac{\beta_k}{\rho_2}\right)
\le \frac{1}{\varepsilon}\|\nabla \mathcal J_{\mathrm{trans}}(\mathbf g)\|_\infty+\frac{\beta_{\max}}{\rho_2},
\]
which concludes the proof.
\end{proof}

\section{Proof of Lemma \ref{lem:global_grad_bound}: Uniform Gradient Bound and Smoothness}\label{app::proof:lem:global_grad_bound}

\textbf{Lemma \ref{lem:global_grad_bound}: Uniform Gradient Bound and Smoothness. }
  On $\mathcal{K}$, the $L_1$-norm of the transport gradient is uniformly bounded:  $\| \nabla \mathcal{J}_{\mathrm{trans}}(\mathbf{g}) \|_1 \le C_{\mathrm{bound}}$, where:
\begin{align}
    C_{\mathrm{bound}}^{\KL} &:=
    \mu(\mathcal{X}) \|\nu\|_1^\alpha \exp\left( \frac{\rho_2 + \delta}{\rho_1 + \varepsilon} \right) \\
    C_{\mathrm{bound}}^{\chi^2}
    &:=\mu(\mathcal X)\frac{\varepsilon}{\rho_1}W\left[\frac{\rho_1}{\varepsilon}e^{\rho_1/\varepsilon}\|\nu\|_1\exp\!\Big(\frac{\rho_2+\delta}{\varepsilon}\Big)\right],
\end{align}

Consequently, the Hessian is bounded on $\mathcal{K}$, and $\mathcal{J}$ is $L$-smooth with $L = \mathcal{O}(1/\varepsilon)$.

\begin{proof}
    In both cases of source divergence, our gradient writes
    \[
    [\nabla \mathcal{J}_{\mathrm{trans}}(\mathbf{g})]_k = \int_{\mathcal{X}} \sigma(x, \mathbf{g}) \, w_k(x, \mathbf{g}) \, d\mu(x)\ , 
    \]
    with $ \sigma_{\KL}(x,\mathbf g):=Z(x,\mathbf g)^\alpha$ or $\sigma_{\chi^2}(x,\mathbf g):=\frac{\varepsilon}{\rho_1}W(U(x,\mathbf g))$ . 
    
    Crucially, we have  $\sum_{k=1}^n w_k(x, \mathbf{g}) = 1$ for any $x$. Because of this property, determining the $L_1$ norm, which is simply the sum of these non-negative components, simplifies easily:
    \[
    \|\nabla \mathcal{J}_{\mathrm{trans}}(\mathbf{g})\|_1 = \sum_{k=1}^n \int_{\mathcal{X}} \sigma(x, \mathbf{g}) \, w_k(x, \mathbf{g}) \, d\mu(x) = \int_{\mathcal{X}} \sigma(x, \mathbf{g}) \underbrace{\left( \sum_{k=1}^n w_k(x, \mathbf{g}) \right)}_{=1} \, d\mu(x).
    \]
    Thus, the problem reduces to finding a uniform bound for the integral of the scalar density $\sigma(x, \mathbf{g})$.
    
    \textbf{Bounding $Z$.}
    The behavior of $\sigma(x, \mathbf{g})$ in both cases is driven by the potential function $Z(x, \mathbf{g}) = \sum_{j} B_j(x, \mathbf{g})$. We recall that:
    \[
    B_j(x, \mathbf{g}) = \beta_j \exp\left( \frac{g_j - c(x,y_j)}{\varepsilon} \right).
    \]
    We can bound this term uniformly by utilizing the problem constraints. Since the cost is non-negative ($c \ge 0$) and the algorithm enforces $g_j \le \rho_2 + \delta $, we have:
    \[
    Z(x, \mathbf{g}) \le \left(\sum_{j=1}^n \beta_j\right) \exp\left( \frac{\rho_2}{\varepsilon} \right) = \|\nu\|_1 \exp\left( \frac{\rho_2 + \delta }{\varepsilon} \right).
    \]
    Let's denote this upper bound constant as $Z_{\max}$.
    
    With $Z$ bounded, we can now bound the total mass $\int \sigma \, d\mu$ for each geometry to conclude:

 \textbf{Case KL:} Here, the density is defined as $\sigma(x, \mathbf{g}) = Z(x, \mathbf{g})^\alpha$. Since $Z(x, \mathbf{g}) \le Z_{\max}$, the gradient norm is directly bounded by:
        \[
        \|\nabla \mathcal{J}_{\mathrm{trans}}\|_1 \le \int_{\mathcal{X}} Z_{\max}^\alpha \, d\mu(x) = \mu(\mathcal{X}) \|\nu\|_1^\alpha \exp\left( \frac{\alpha (\rho_2+ \delta)}{\varepsilon} \right).
        \]
        Substituting $\alpha = \frac{\varepsilon}{\rho_1 + \varepsilon}$ yields the final bound $\mu(\mathcal{X}) \|\nu\|_1^\alpha \exp\left( \frac{\rho_2 + \delta }{\rho_1 + \varepsilon} \right)$.

 \textbf{Case $\chi^2$:}For the $\chi^2$-source semi-dual obtained by eliminating $f$, the transport-gradient density is
\[
\sigma(x,\mathbf g)=\frac{\varepsilon}{\rho_1}W(U(x,\mathbf g)),
\qquad
U(x,\mathbf g)=\frac{\rho_1}{\varepsilon}Z(x,\mathbf g)e^{\rho_1/\varepsilon}.
\]
Since $W$ is increasing on $\mathbb R_+$ and $Z\le Z_{\max}$, we have $U(x,\mathbf g)\le U_{\max}$ with
\[
U_{\max}:=\frac{\rho_1}{\varepsilon}Z_{\max}e^{\rho_1/\varepsilon}
=\frac{\rho_1}{\varepsilon}\,\|\nu\|_1\,\exp\!\Big(\tfrac{\rho_2+\delta}{\varepsilon}\Big).
\]
Therefore,
\[
\|\nabla \mathcal{J}_{\mathrm{trans}}(\mathbf{g})\|_1
=\int_{\mathcal X}\frac{\varepsilon}{\rho_1}W(U(x,\mathbf g))\d\mu(x)
\le \mu(\mathcal X)\,\frac{\varepsilon}{\rho_1}\,W(U_{\max}),
\]
which is a uniform bound under the constraint $g_j\le \rho_2 + \delta$. Moreover, one can check numerically that this constant is close to $1$. 
\end{proof}

\subsection{Corollary - Proof of Lemma \ref{lemma::variance} : Bounded Variance}\label{app::proof_lemma::variance}

\textbf{Proposition \ref{lemma::variance} :Variance Bound of Mini-Batch Gradient.} 
Let $\widehat{\nabla} \mathcal{J}(\mathbf{g})$ be the mini-batch gradient estimator computed with batch size $b\geq 1$, as defined in Eq. \eqref{eq:stoch_grad_estimator}.
For any $\mathbf{g} \in \mathcal{K}$, using the uniform bound $C_{\mathrm{bound}}$ from Lemma \ref{lem:global_grad_bound}, the variance is bounded by:
\begin{equation}
    \mathbb{E}\left[ \| \widehat{\nabla} \mathcal{J}(\mathbf{g}) - \nabla \mathcal{J}(\mathbf{g}) \|_2^2 \right] 
    \le 
    \frac{4 C_{\mathrm{bound}}^2}{b}\ . 
\end{equation}

\begin{proof}
From Lemma \ref{lem:global_grad_bound}, we have the uniform $L_1$-norm bound on the estimation error for any sample realization:
\[
    \| \widehat{\nabla} \mathcal{J}(\mathbf{g}) - \nabla \mathcal{J}(\mathbf{g}) \|_1 \le 2 C_{\mathrm{bound}}.
\]
Using the norm inequality $\|\cdot\|_2 \le \|\cdot\|_1$, the squared Euclidean error for a single sample ($b=1$) is bounded almost surely by $(2 C_{\mathrm{bound}})^2$. 
Since $\widehat{\nabla} \mathcal{J}(\mathbf{g})$ is the average of $b$ i.i.d. estimators, the variance of the mean scales by $1/b$:
\[
    \mathbb{E}\left[ \| \widehat{\nabla} \mathcal{J}(\mathbf{g}) - \nabla \mathcal{J}(\mathbf{g}) \|_2^2 \right] 
    = \frac{1}{b} \text{Var}(\widehat{\nabla}_{\text{single}}) 
    \le \frac{(2 C_{\mathrm{bound}})^2}{b}.
\]
\end{proof}

\section{Proof of Proposition \ref{prop::self_concordance}: Generalized self-concordance of the semi-dual}\label{app::proof_self_conc}

\textbf{Proposition \ref{prop::self_concordance} : Generalized self-concordance. }The semi-dual $\J$ is generalized self-concordant. That is, for $M = \frac{2 + 3\alpha}{\varepsilon}$ for $\KL$ source, and $M = \frac{6}{\varepsilon}$ for $\chi^2$ , we have for any $\bfg \in \mathbb{R}^n$ and any direction $\bfh \in \mathbb{R}^n$:
    \[
        \big| \nabla^3 \J(\bfg)[\bfh, \bfh, \bfh] \big| \le M \|\bfh\|_\infty \, \langle \bfh, \nabla^2 \J(\bfg) \bfh \rangle.
    \]

\subsection{Case 1: \texorpdfstring{$D_1 = \mathrm{KL},\ D_2 = \chi^2$}{D₁ = KL, D₂ = χ²}}

\begin{proof}
For clarity, we recall the notations
\[
B_j(\bfg):=\beta_j\exp\!\left(\frac{g_j-c(x,y_j)}{\varepsilon}\right),\qquad
Z(\bfg):=\sum_{j=1}^n B_j(\bfg),\qquad
\tau(\bfg):=Z(\bfg)^\alpha,
\]
with $\beta_j>0$. The softmax weights are 
\[
w_j(\bfg):=\frac{B_j(\bfg)}{Z(\bfg)}\in\Delta^n.
\]

We now introduce some notions regarding directional derivatives: 

For a direction $\bfh\in\mathbb R^n$, denote directional derivatives by $\partial_{\bfh}$ and define
\[
w_{\bfh}:=\langle w,\bfh\rangle,\qquad
w_{\bfh^2}:=\langle w,\bfh^2\rangle,\qquad
w_{\bfh^3}:=\langle w,\bfh^3\rangle,
\]
where $\bfh^2=(h_1^2,\dots,h_n^2)$ and $\bfh^3=(h_1^3,\dots,h_n^3)$. Let $L:=\|\bfh\|_\infty$.

We start by giving the directional derivatives of $\tau$. 

\textbf{Derivatives of $Z$: } Consider $\bfg(t)=\bfg+t\bfh$. Then
\[
Z(t)=\sum_{j=1}^n B_j(\bfg)\exp\!\left(\frac{t h_j}{\varepsilon}\right).
\]
Differentiating at $t=0$ yields
\[
\partial_{\bfh} Z
= Z'(0)=\frac{1}{\varepsilon}\sum_j B_j(\bfg)h_j
=\frac{1}{\varepsilon}Z(\bfg)\,w_{\bfh},
\]
\[
\partial_{\bfh}^2 Z
= Z''(0)=\frac{1}{\varepsilon^2}\sum_j B_j(\bfg)h_j^2
=\frac{1}{\varepsilon^2}Z(\bfg)\,w_{\bfh^2},
\]
\[
\partial_{\bfh}^3 Z
= Z'''(0)=\frac{1}{\varepsilon^3}\sum_j B_j(\bfg)h_j^3
=\frac{1}{\varepsilon^3}Z(\bfg)\,w_{\bfh^3}.
\]

\textbf{Derivatives of $\tau = Z^\alpha$:} Using the chain rule for $\tau(t)=Z(t)^\alpha$,
\[
\tau'=\alpha Z^{\alpha-1}Z',\qquad
\tau''=\alpha(\alpha-1)Z^{\alpha-2}(Z')^2+\alpha Z^{\alpha-1}Z'',
\]
\[
\tau'''=\alpha(\alpha-1)(\alpha-2)Z^{\alpha-3}(Z')^3
+3\alpha(\alpha-1)Z^{\alpha-2}Z'Z''
+\alpha Z^{\alpha-1}Z'''.
\]
Substituting the expressions from the derivatives of $Z$, and writing $\tau=Z^\alpha$, we get:
\[
\partial_{\bfh}\tau
=\frac{\alpha}{\varepsilon}\tau\,w_{\bfh},
\]
\[
\partial_{\bfh}^2\tau
=\frac{\alpha}{\varepsilon^2}\tau\Big(w_{\bfh^2}-(1-\alpha)w_{\bfh}^2\Big).
\]
Define
\[
D:=w_{\bfh^2}-(1-\alpha)w_{\bfh}^2.
\]
Then
\[
\partial_{\bfh}^2\tau=\frac{\alpha}{\varepsilon^2}\tau D.
\]
Moreover, $D\ge 0$ since
\[
D=(w_{\bfh^2}-w_{\bfh}^2)+\alpha w_{\bfh}^2
=\mathrm{Var}_w(H)+\alpha(\mathbb E_w H)^2\ge 0,
\]
where $H$ is the random variable taking values $h_j$ with probabilities $w_j$.

A direct substitution into the third-derivative formula also gives
\[
\partial_{\bfh}^3\tau
=\frac{\alpha}{\varepsilon^3}\tau\Big(
w_{\bfh^3}+3(\alpha-1)w_{\bfh}w_{\bfh^2}+(\alpha-1)(\alpha-2)w_{\bfh}^3
\Big).
\]

\textbf{Central-moment rewrite:} As in \cite{bercu2021asymptotic}, we will substitute moments into  the expansion of the third-derivative.  

Let $H$ be the random variable with $\mathbb P(H=h_j)=w_j$. Define
\[
m:=\mathbb E[H]=w_{\bfh},\qquad
\sigma^2:=\mathrm{Var}(H)=w_{\bfh^2}-w_{\bfh}^2,\qquad
\kappa_3:=\mathbb E[(H-m)^3].
\]
Then $\mathbb E[H^2]=\sigma^2+\mu^2$ and $\mathbb E[H^3]=\kappa_3+3\mu\sigma^2+m^3$.
With these identities, one checks that the bracket in $\partial_{\bfh}^3\tau$ equals
\[
N:=\kappa_3+3\alpha m\sigma^2+\alpha^2m^3.
\]
Hence
\[
\partial_{\bfh}^3\tau=\frac{\alpha}{\varepsilon^3}\tau\,N,
\qquad
D=\sigma^2+\alpha m^2.
\]

\textbf{Bounding $|N|$ by $D$:} Let $L_h =\|\bfh\|_\infty$. Since $|H|\le L_h$ almost surely, we have $|m|\le L_h$ and also
$|H-m|\le |H|+|m|\le 2L_h$ almost surely. Therefore,
\[
|\kappa_3|
=\big|\mathbb E[(H-m)^3]\big|
\le \mathbb E[|H-m|^3]
\le 2L\,\mathbb E[(H-m)^2]
=2L_h\,\sigma^2.
\]
Using $|m|\le L_h$ and $\alpha\in(0,1]$:
\begin{align*}
|N|
&\le |\kappa_3|+3\alpha|m|\sigma^2+\alpha^2|m|^3 \\
&\le 2L_h\sigma^2+3\alpha L_h\sigma^2+\alpha^2L_h\mu^2 \\
&\le L_h\Big((2+3\alpha)\sigma^2+\alpha m^2\Big)\\
&\le (2+3\alpha)L_h(\sigma^2+\alpha m^2) \\
&=(2+3\alpha)\|\bfh\|_\infty D.
\end{align*}

Combining $\partial_{\bfh}^2\tau=\frac{\alpha}{\varepsilon^2}\tau D$ and
$\partial_{\bfh}^3\tau=\frac{\alpha}{\varepsilon^3}\tau N$ with the bound $|N|\le (2+3\alpha)\|\bfh\|_\infty D$,
we obtain
\[
|\partial_{\bfh}^3\tau|
=\frac{\alpha}{\varepsilon^3}\tau|N|
\le \frac{\alpha}{\varepsilon^3}\tau\,(2+3\alpha)\|\bfh\|_\infty D
=\frac{2+3\alpha}{\varepsilon}\,\|\bfh\|_\infty\,\partial_{\bfh}^2\tau.
\]
This proves that $\tau$ is quasi-self-concordant with parameter
\[
M=\frac{2+3\alpha}{\varepsilon}.
\]

By differentiation under the integral, from the boundedness of the measure $\mu$, integrating the pointwise inequality yields
\[
|\partial_{\bfh}^3\mathcal J(\bfg)|
\le \frac{2+3\alpha}{\varepsilon}\,\|\bfh\|_\infty\,\partial_{\bfh}^2\mathcal J(\bfg).
\]
\end{proof}

\subsection{Case 2: \texorpdfstring{$D_1 = D_2 = \chi^2$}{D₁ = D₂ = χ²}}

\begin{proof}
We keep the same notations as in Case 1:
\[
B_j(\bfg)=\beta_j\exp\!\left(\frac{g_j-c(x,y_j)}{\varepsilon}\right),\qquad
Z(\bfg)=\sum_{j=1}^n B_j(\bfg),\qquad
w_j(\bfg)=\frac{B_j(\bfg)}{Z(\bfg)}\in\Delta^n,
\]
and for a direction $\bfh\in\mathbb R^n$ we denote
\[
m:=w_{\bfh}=\langle w,\bfh\rangle,\qquad
w_{\bfh^2}=\langle w,\bfh^2\rangle,\qquad
\sigma^2:=w_{\bfh^2}-m^2,\qquad
\kappa_3:=\E[(H-m)^3],\qquad
L:=\|\bfh\|_\infty,
\]
where $H$ takes values $h_j$ with probabilities $w_j$. Along $\bfg(t)=\bfg+t\bfh$, the directional derivatives of $Z$
are unchanged:
\[
\partial_{\bfh} Z=\frac{1}{\varepsilon}Z\,m,\qquad
\partial_{\bfh}^2 Z=\frac{1}{\varepsilon^2}Z\,w_{\bfh^2},\qquad
\partial_{\bfh}^3 Z=\frac{1}{\varepsilon^3}Z\,w_{\bfh^3}.
\]
We also reuse the same probabilistic bound as in Case 1: since $|H|\le L$ a.s., one has
\begin{equation}\label{eq:kappa3-bound-trans}
|\kappa_3|\le \E[|H-m|^3]\le 2L\,\E[(H-m)^2]=2L\,\sigma^2.
\end{equation}

\medskip
\noindent\textbf{Specific changes.}
For the $\chi^2$-source semi-dual obtained by eliminating $f$, the \emph{transport} integrand is
\[
\tau_{\chi^2}(\bfg)
:=\frac{\varepsilon^2}{\rho_1}\Big(W(U(\bfg))+\tfrac12 W(U(\bfg))^2\Big),
\qquad
U(\bfg):=\frac{\rho_1}{\varepsilon}e^{\rho_1/\varepsilon}Z(\bfg).
\]
Equivalently, $\tau_{\chi^2}=\psi(Z)$ with
\[
\psi(z):=\frac{\varepsilon^2}{\rho_1}\Big(W(az)+\tfrac12 W(az)^2\Big),\qquad
a:=\frac{\rho_1}{\varepsilon}e^{\rho_1/\varepsilon}.
\]
Write $\omega:=W(az)\ge 0$ (so $\omega e^\omega=az$). Using $W'(u)=\frac{W(u)}{u(1+W(u))}$ and the identity
$u=az$, one obtains the explicit derivatives
\begin{equation}\label{eq:psi-derivatives-new}
\psi'(z)=\frac{\varepsilon^2}{\rho_1}\frac{\omega}{z},\qquad
\psi''(z)=-\frac{\varepsilon^2}{\rho_1}\frac{\omega^2}{z^2(1+\omega)},\qquad
\psi'''(z)=\frac{\varepsilon^2}{\rho_1}\frac{\omega^3(3+2\omega)}{z^3(1+\omega)^3}.
\end{equation}

\medskip
\noindent\textbf{Second derivative along $\bfh$.}
By the one-dimensional chain rule,
\[
\partial_{\bfh}^2\tau_{\chi^2}
=\psi''(Z)(\partial_{\bfh}Z)^2+\psi'(Z)\partial_{\bfh}^2Z.
\]
Substituting \eqref{eq:psi-derivatives-new} and the derivatives of $Z$ yields
\begin{align}
\label{eq:tau2-new}
\partial_{\bfh}^2\tau_{\chi^2}
&=
\left(-\frac{\varepsilon^2}{\rho_1}\frac{\omega^2}{Z^2(1+\omega)}\right)\left(\frac{Z^2 m^2}{\varepsilon^2}\right)
+\left(\frac{\varepsilon^2}{\rho_1}\frac{\omega}{Z}\right)\left(\frac{Z w_{\bfh^2}}{\varepsilon^2}\right)\nonumber\\
&=\frac{1}{\rho_1}\left(\omega w_{\bfh^2}-\frac{\omega^2}{1+\omega}m^2\right)
=\frac{\omega}{\rho_1}\left(\sigma^2+\frac{m^2}{1+\omega}\right)\ \ge 0.
\end{align}

\medskip
\noindent\textbf{Third derivative along $\bfh$ and central-moment simplification.}
Similarly,
\[
\partial_{\bfh}^3\tau_{\chi^2}
=\psi'''(Z)(\partial_{\bfh}Z)^3+3\psi''(Z)\partial_{\bfh}Z\,\partial_{\bfh}^2Z+\psi'(Z)\partial_{\bfh}^3Z.
\]
Substituting \eqref{eq:psi-derivatives-new} and the derivatives of $Z$ gives
\[
\partial_{\bfh}^3\tau_{\chi^2}
=\frac{1}{\rho_1\varepsilon}\left[
\omega w_{\bfh^3}
-\frac{3\omega^2}{1+\omega}m w_{\bfh^2}
+\frac{\omega^3(3+2\omega)}{(1+\omega)^3}m^3
\right].
\]
Using $w_{\bfh^2}=\sigma^2+m^2$ and $w_{\bfh^3}=\kappa_3+3m\sigma^2+m^3$, the cubic terms in $m^3$
yield the identity
\begin{equation}
\label{eq:tau3-new}
\partial_{\bfh}^3\tau_{\chi^2}
=\frac{\omega}{\rho_1\varepsilon}\left[
\kappa_3+\frac{3m}{1+\omega}\sigma^2 + \frac{m^3}{(1 + \omega)^3}
\right].
\end{equation}

\medskip
\noindent\textbf{Generalized self-concordance bound.}
Using \eqref{eq:kappa3-bound-trans}, $|m|\le L$, and $(1+\omega)^{-1}\le 1$, we get
\[
\left|\kappa_3+\frac{3m}{1+\omega}\sigma^2\right|
\le |\kappa_3|+3|m|\sigma^2
\le 2L\sigma^2+3L\sigma^2
=5L\sigma^2.
\]
Moreover, since $(1+\omega)^{-3}\le (1+\omega)^{-1}$ and $|m|\le L$, we also have
\[
\left|\frac{m^3}{(1+\omega)^3}\right|
\le \frac{|m|}{1+\omega}\,m^2
\le L\,\frac{m^2}{1+\omega}.
\]
Combining these two bounds with \eqref{eq:tau3-new} yields
\[
|\partial_{\bfh}^3\tau_{\chi^2}|
\le \frac{\omega}{\rho_1\varepsilon}\,L\left(5\sigma^2+\frac{m^2}{1+\omega}\right).
\]
Since $5\sigma^2+\frac{m^2}{1+\omega}\le 6\left(\sigma^2+\frac{m^2}{1+\omega}\right)$ and using \eqref{eq:tau2-new},
\[
\partial_{\bfh}^2\tau_{\chi^2}
=\frac{\omega}{\rho_1}\left(\sigma^2+\frac{m^2}{1+\omega}\right),
\]
we obtain the pointwise inequality
\[
|\partial_{\bfh}^3\tau_{\chi^2}|
\le \frac{6}{\varepsilon}\,L\,\partial_{\bfh}^2\tau_{\chi^2}.
\]
Therefore, $\tau_{\chi^2}$ is quasi-self-concordant with parameter $M=6/\varepsilon$.
By differentiation under the integral, the same bound transfers to the transport functional
$\mathcal J_{\mathrm{trans}}(\bfg)=\int_{\mathcal X}\tau_{\chi^2}(x,\bfg)\d\mu(x)$:
\[
|\partial_{\bfh}^3\mathcal J_{\mathrm{trans}}(\bfg)|
\le \frac{6}{\varepsilon}\,\|\bfh\|_\infty\,\partial_{\bfh}^2\mathcal J_{\mathrm{trans}}(\bfg).
\]

\end{proof}

As a corollary of generalized self-concordance, we have enhanced control over the Hessian; see, for instance, Proposition 8 in \cite{sun2019generalized}. However, here, this is with respect to the infinity norm instead of the Euclidean norm.

\bigskip
\begin{corollary}\label{cor::self_conc_hessian}
    Noting $M = \frac{2+3\alpha}{\varepsilon}$ for the $\KL$ source case,  $M = \frac{6}{\varepsilon}$ else, we have for any $\bfg_, \bfg_2 \in \RR^n :$ 
    $$e^{-M \|\bfg_2 - \bfg_1\|_{\infty}} \nabla^2 \J(\bfg_1) \preceq \nabla^2 \J (\bfg_2) \preceq e^{M \|\bfg_2 - \bfg_1\|_{\infty}} \nabla^2 \J(\bfg_1)$$
\end{corollary}

\bigskip

\section{Proof of Proposition~\ref{prop:asym_smoothness_line}: \texorpdfstring{$(L_0, L_1)$}{(L₀, L₁)}-smoothness}

\begin{proof}
Fix $g\in\mathbb{R}^n$ and consider the segment
\[
g_s := g - s\lambda \nabla J(g), \qquad s\in[0,1],
\]
with $\lambda>0$. Then $\|g_s-g\|_\infty = s\lambda \|\nabla J(g)\|_\infty$. 

\noindent By generalized self-concordance (Corollary~\ref{cor::self_conc_hessian} stated with $\|\cdot\|_\infty$), we have
\[
\nabla^2 J(g_s) \preceq \exp\!\big(M \|g_s-g\|_\infty\big)\,\nabla^2 J(g)
= \exp\!\big(M s\lambda \|\nabla J(g)\|_\infty\big)\,\nabla^2 J(g).
\]

\noindent Choose the step size as $\lambda := 1/\widetilde L(g)$ where
\[
\widetilde L(g)
:=
e\left(\frac{1}{\varepsilon}\|\nabla J_{\mathrm{trans}}(g)\|_\infty + \frac{\beta_{\max}}{\rho_2}\right)
\;+\;
M\|\nabla J(g)\|_\infty.
\]
Then $M\lambda \|\nabla J(g)\|_\infty \le 1$, hence for all $s\in[0,1]$,
\[
\nabla^2 J(g_s) \preceq e\,\nabla^2 J(g),
\quad\text{and thus}\quad
\sup_{s\in[0,1]}\|\nabla^2 J(g_s)\|_{\mathrm{op}}
\le e\,\|\nabla^2 J(g)\|_{\mathrm{op}}.
\]

Next, we upper bound $\|\nabla^2 J(g)\|_{\mathrm{op}}$ by the local smoothness proxy computed from
$\nabla J_{\mathrm{trans}}(g)$. By Theorem~\ref{thm:smoothness_unified},
\[
\|\nabla^2 J(g)\|_{\mathrm{op}}
\le \frac{1}{\varepsilon}\|\nabla J_{\mathrm{trans}}(g)\|_\infty + \frac{\beta_{\max}}{\rho_2}.
\]
Therefore,
\[
\sup_{s\in[0,1]}\|\nabla^2 J(g_s)\|_{\mathrm{op}}
\le
e\left(\frac{1}{\varepsilon}\|\nabla J_{\mathrm{trans}}(g)\|_\infty + \frac{\beta_{\max}}{\rho_2}\right)
\le \widetilde L(g).
\]

Finally, for any two points $g_1,g_2$ on the segment $\{g_s:s\in[0,1]\}$, the mean value theorem yields
\[
\|\nabla J(g_1)-\nabla J(g_2)\|
\le
\left(\sup_{s\in[0,1]}\|\nabla^2 J(g_s)\|_{\mathrm{op}}\right)\|g_1-g_2\|
\le
\widetilde L(g)\,\|g_1-g_2\|.
\]
This concludes the proof.
\end{proof}

\section{Proof of Theorem \ref{thm:pasgd_convergence} : Convergence of PASGD}
\label{app:proof_pasgd}

\textbf{Theorem \ref{thm:pasgd_convergence}: Convergence of PASGD.}
Let the step sizes be chosen as $\eta_t = C t^{-\gamma}$ with $\gamma \in (1/2, 1)$. Under Setting \ref{setting::semi_discrete} and the projection onto $\mathcal{K}$, the averaged iterate $\bar{\mathbf{g}}_T$ converges to the optimum $\mathbf{g}^\star$ in objective value with an expected error of:
\[
\mathbb{E}\left[ \mathcal{J}(\bar{\mathbf{g}}_T) - \mathcal{J}(\mathbf{g}^\star) \right] = \mathcal{O}\left( \frac{n\rho_2^2}{\varepsilon T} \right).
\]

\begin{proof}
We study the projected SGD recursion
\[
g_{t+1}=\Pi_{\K}(g_t-\gamma_t \widehat{\nabla}\J(g_t)),
\qquad 
\bar g_T=\frac1T\sum_{t=1}^T g_t,
\qquad 
\gamma_t=Ct^{-\beta},\ \beta\in(1/2,1).
\]
For the non-averaged iterates, the projection brings no additional difficulty: $\Pi_{\K}$ is $1$-Lipschitz, so all standard
Robbins--Monro estimates based on a one-step expansion remain valid. We therefore rely on the results of \citet{gadat2017optimal},
and we start by verifying their martingale increment assumption.

\textbf{Verification of $(\mathrm{HSC}_{\Sigma_p})$ for all $p$:}
Let
\[
\xi_{t+1}:=\widehat{\nabla}\J(g_t)-\nabla\J(g_t),
\qquad 
\F_t:=\sigma(g_0,X_1,\dots,X_t)
\]
be the noise and the natural filtration. We assume (by construction of the algorithm) that $g_t\in\K$ for all $t$.

For every integer $p\ge 1$, the condition $(\mathrm{HSC}_{\Sigma_p})$ of \citet[Sec.~1.3.3]{gadat2017optimal} holds with
\[
\Sigma_p:=(2C_{\mathrm{bound}})^{2p}.
\]

Indeed, by Lemma~\ref{lem:global_grad_bound}, on $\K$ the transport gradient has a uniform $\ell_1$ bound:
for any $g\in\K$ and any realization of the mini-batch,
\[
\|\widehat{\nabla}\J_{\trans}(g)\|_1\le C_{\mathrm{bound}},
\qquad
\|\nabla\J_{\trans}(g)\|_1\le C_{\mathrm{bound}}.
\]
Since the quadratic term in $\nabla\J$ is deterministic, the noise satisfies
\[
\|\widehat{\nabla}\J(g)-\nabla\J(g)\|_1
=
\|\widehat{\nabla}\J_{\trans}(g)-\nabla\J_{\trans}(g)\|_1
\le \|\widehat{\nabla}\J_{\trans}(g)\|_1+\|\nabla\J_{\trans}(g)\|_1
\le 2C_{\mathrm{bound}}.
\]
Using $\|\cdot\|_2\le \|\cdot\|_1$, we obtain $\|\xi_{t+1}\|_2\le 2C_{\mathrm{bound}}$ almost surely, hence for all $p\ge 1$,
\[
\EE\!\left[\|\xi_{t+1}\|_2^{2p}\mid \F_t\right]\le (2C_{\mathrm{bound}})^{2p}=\Sigma_p.
\]
Finally, since $1+\J(g_t)^p\ge 1$, we also have
\[
\EE\!\left[\|\xi_{t+1}\|_2^{2p}\mid \F_t\right]\le \Sigma_p\bigl(1+\J(g_t)^p\bigr),
\]
which is exactly $(\mathrm{HSC}_{\Sigma_p})$ in the sense of \citet{gadat2017optimal}.

As a consequence, Proposition~1.1 in \citet{gadat2017optimal} applies and yields, for any $p\ge 1$,
\begin{equation}
\label{eq:gp_nonavg_moments}
\EE\big[\|g_t-g^\star\|^{2p}\big]\ \lesssim\ \gamma_t^{p}.
\end{equation}

\textbf{The projection is asymptotically negligible:} While the projection is harmless for the analysis of the non-averaged recursion, it is convenient to show that it becomes
asymptotically inactive. Fix $\delta=1$ and denote $\K=\K_1=\{g:\ g_k\le \rho_2+1\}$.
Consider the event that projection is active at time $t$:
\[
\mathcal{P}_t:=\Big\{g_t-\gamma_t\widehat{\nabla}\J(g_t)\notin \K\Big\}.
\]
Since $g^\star\in\K_0$ (Proposition~\ref{prop:optimality_constraints}) and $\K_0\subset\K$, this event can only happen when
$g_t$ is sufficiently far from $g^\star$. In particular, there exists a constant $r>0$ (depending on $\rho_2$ and $\delta$ only)
such that $\mathcal{P}_t\subset\{\|g_t-g^\star\|_\infty\ge r\}$ for all $t$ large enough.
Therefore, for any integer $p\ge 1$, Markov's inequality and \eqref{eq:gp_nonavg_moments} give
\[
\PP(\mathcal{P}_t)
\le \PP(\|g_t-g^\star\|_\infty\ge r)
\le \frac{\EE\|g_t-g^\star\|^{2p}}{r^{2p}}
\ \lesssim\ \frac{\gamma_t^{p}}{r^{2p}}.
\]
Since $p$ is arbitrary and $\gamma_t=t^{-\beta}$ with $\beta\in(1/2,1)$, we can make $\PP(\mathcal{P}_t)=o(t^{-a})$
for any prescribed $a>0$ by choosing $p$ large enough. Therefore, the projection affects a vanishing fraction of iterates.

\textbf{Averaged iterates and objective error:} We can now invoke Corollary~1.1 in \citet{gadat2017optimal}, which gives the standard Polyak--Ruppert behavior for $\bar g_T$. Noting $\mathbf{H} = \nabla^2 \J(\bfg^\star)$, we have :
\begin{equation}\label{eq:avg_param_rate}
    \mathbb{E}[\|\bar{\mathbf{g}}_T - \mathbf{g}^\star\|^2] \le \frac{\text{Tr}(\mathbf{H}^{-1}\mathbf{\Sigma}\mathbf{H}^{-1})}{T} + o(1/T).
\end{equation}
Given the bounded noise variance $\mathbf{\Sigma}$ (coming from Lemma \ref{lem:global_grad_bound}), this term scales as $\mathcal{O}(\rho_2^2 n^2/T)$.

We now convert \eqref{eq:avg_param_rate} into an objective bound using the generalized self-concordance of the semi-dual (Proposition \ref{prop::self_concordance}).
Locally, the smoothness at $g^\star$ satisfies $L(g^\star)\asymp 1/(n\varepsilon)$ (Corollary~\ref{cor:local_conditioning}),
and Corollary~\ref{cor::self_conc_hessian} ensures that along a neighborhood of $g^\star$,
\[
L(\bar \bfg_T)\ \le\ \exp\!\Big(M\|\bar g_T-g^\star\|_\infty\Big)\,L(\bfg^\star),
\qquad M=\frac{2+3\alpha}{\varepsilon}\ \ \text{(KL-source)}.
\]
Splitting on the event $\{\|\bar g_T-g^\star\|\le \varepsilon\}$ and using the above local control yields
\begin{align*}
\EE\big[\J(\bar \bfg_T)-\J^\star\big]
\;\le\;
\frac{e^{M\varepsilon}L(g^\star)}{2}\,\EE\big[\|\bar \bfg_T-\bfg^\star\|^2\big]
\;+\;
\frac{C_1}{\varepsilon}\,\PP\big(\|\bar \bfg_T-\bfg^\star\|\ge \varepsilon\big),
\end{align*}
for a finite constant $C_1$ (depending only on $\K$ through the crude curvature bound on $\K$).

Finally, using Theorem 4.4 in \cite{godichon2019lp}, we have high-order moment of averaged iterates
\begin{equation}
\label{eq:gp_avg_moments}
\EE\big[\|\bar{\bfg}_t-\bfg^\star\|^{2p}\big]\ \lesssim\ \frac{1}{t^p}.
\end{equation}
Again, using Markov's inequality with high order moment gives 
\[
\PP\big(\|\bar g_T-g^\star\|\ge \varepsilon\big) = o(t^{-a})
\]
for all $a$, which concludes.

\end{proof}
\paragraph{Remark.}
This adaptivity argument uses a step-size schedule $\gamma_t\propto t^{-\beta}$ with $\beta<1$ to ensure good control of higher
moments and tail probabilities, which would have not been possible with the non-averaged iterates, using $\gamma_t \propto \tfrac{1}{t}$.

\section{Proof of Theorem  \ref{th::anag_cv}: Adaptive NAG Convergence Rate}\label{app::proof_th::anag_cv}

\textbf{Theorem \ref{th::anag_cv} : Adaptive NAG Convergence Rate.} 
     Let $R$ be the number of restart. Then, the iterates generated by Algorithm \ref{alg:adaptive_nag} satisfy
    \begin{equation*}
    \mathcal{J}(g_{T+1}) - \mathcal{J}^\star
    \leq
    2^{R}\left( \mathcal{J}(g_0) -\mathcal{J}^\star + \frac{\beta_{\min}}{2\rho_2}\|g_0 - g^\star\|^2\right)
    \prod_{t=0}^{T}\left( 1 - \sqrt{\frac{\beta_{\min}}{\rho_2 L_t}}\right),
    \end{equation*}
    
    \noindent Furthermore, the algorithm ensures $y_t \in K_1$ for all $t$, so using $C_{\text{bound}}$ from Lemma \ref{lem:global_grad_bound}, we have $L_t \leq \bar{L} = \mathcal{O}(C_{\text{bound}}/\varepsilon)$ for all $t$. This implies the following rates:
    \begin{enumerate}
        \item \textbf{Global Rate:}  We have at least the contraction rate $1 - \mathcal{O}(\sqrt{\varepsilon/\beta_{\min}\rho_2})$ for both the objective gap and gradient norm.
        \item \textbf{Local Rate:} Since $L_t \leq 2\beta_{\max}\left(\frac{1}{\varepsilon} + \frac{1}{\rho_2} \right) + 2\|\nabla \mathcal{J}(g_{T+1})\|$, assuming $\beta_{\min}/\beta_{\max} \simeq 1$ and substituting the contraction rate of the gradient norm into \eqref{eq:nag_adaptive_rate} shows that, locally, we have a contraction rate of $1 - \mathcal{O}(\sqrt{\varepsilon/\rho_2})$.
    \end{enumerate}

\begin{proof}

\textbf{Contraction of the potential function when we do not restart.}
The restart schemes permit us to stay in the region $K_{\delta}$, where we fix $\delta = 1$. We analyze here the contraction, when restart is does not happen. 

Following the Lyapunov analysis for accelerated methods \citep{nesterov2015universal, bansal2019potential}, we define the potential function at iteration $t$ as:
\begin{equation}
    \Psi_t \triangleq \frac{1}{\sqrt{L_{t-1}}} \left( \mathcal{J}(g_t) - \mathcal{J}^\star + \frac{\alpha}{2} \|z_t - g^\star\|^2 \right),
\end{equation}
where $z_t := \bfg_t + \left(\sqrt{L_t/\alpha} - 1\right)(\bfg_t - \bfg_{t-1})$ is the auxiliary sequence, $\alpha$ is the strong convexity parameter of our function (here = $\frac{\beta_{\min}}{\rho_2}$), $L_t$ is the smoothness bound on the segment $[y_t, \bfg_{t+1}]$ and by convention $L_{-1} = L_0$.

In \cite{bansal2019potential}, the proof there bounds $\Delta\Psi_t$ by combining (i) one smoothness-based decrease inequality
for the gradient step and (ii) an algebraic expansion of the quadratic term in the potential
(\cite[eqs.~(5.25)--(5.27)]{bansal2019potential}). 
In our constrained case, the algebraic part is unchanged; only (i) changes.

Indeed, our update is the projected step
$g_{t+1}=\Pi_{\mathcal K}(y_t-\tfrac{1}{L_t}\nabla \mathcal J(y_t))$.
Define the gradient-mapping residual $\Delta_t := L_t(y_t-g_{t+1})$.
By Proposition~\ref{prop:asym_smoothness_line}, $\mathcal J$ is $L_t$-smooth on the segment
between $y_t$ and $g_{t+1}$, and by optimality of the projection we obtain the projected descent
inequality
\[
\mathcal J(g_{t+1})
\le \mathcal J(y_t)-\frac{1}{2L_t}\|\Delta_t\|^2,
\]
which replaces the unconstrained inequality
$f(y_{t+1})\le f(x_t)-\frac{1}{2\beta}\|\nabla_t\|^2$ used in \cite[p.~28]{bansal2019potential}, where $\nabla_t = \nabla f(x_t)$ with their notations.

With this substitution (replace $\nabla_t$ by $\Delta_t$ and $\beta$ by $L_t$),
the remainder of the potential-change calculation is identical to
\cite[eqs.~(5.25)--(5.27)]{bansal2019potential}, accounting here for the difference between $L_{t-1}$ and $L_t$, yielding the one-step contraction,
\[
\Psi_{t+1} \le \sqrt{\frac{L_{t-1}}{L_t}}\left(1-\sqrt{\frac{\alpha}{L_t}}\right)\Psi_t,
\]
and unrolling gives
\[
\mathcal J(g_{T+1})-\mathcal J^\star
\le \Phi_0\prod_{t=0}^T\left(1-\sqrt{\frac{\alpha}{L_t}}\right).
\]

\textbf{Taking into account the restarts.}
We implement a safeguard restart to ensure that all smoothness arguments are made inside the bounded region
$\mathcal K_1$. Concretely, at any iteration $t$ such that the extrapolated point leaves the safe set,
$y_t\notin\mathcal K_1$, we restart by resetting the acceleration state (zeroing the momentum):
\[
y_t \leftarrow g_t,\qquad g_{t-1}\leftarrow g_t,
\]
so that the auxiliary variable definition $z_t = g_t + (\sqrt{L_t/\alpha}-1)(g_t-g_{t-1})$ is consistent and implies
$z_t=g_t$. We also re-initialize the scalar parameters of the Lyapunov construction for the new epoch:  we start a new epoch with local index $s=t$ and enforce $L_{s-1}=L_s$ by convention.
The projected update
$g_{t+1}=\Pi_{\mathcal K}(y_t-\tfrac{1}{L_t}\nabla \mathcal J(y_t))$ is then performed from $y_t=g_t\in\mathcal K_{0.1}\subset\mathcal K_1$,
so that Proposition~\ref{prop:asym_smoothness_line} applies on the segment $[y_t,g_{t+1}]$.

Define the unnormalized potential
\[
E_t :=\sqrt{L_{t-1}}\Psi_t
= \mathcal{J}(g_t) - \mathcal{J}^\star + \frac{\alpha}{2} \|z_t - g^\star\|^2.
\]
Multiplying the one-step inequality
\[
\Psi_{t+1} \le \sqrt{\frac{L_{t-1}}{L_t}}\left(1-\sqrt{\frac{\alpha}{L_t}}\right)\Psi_t
\]
by $\sqrt{L_t}$ yields the clean contraction
\[
E_{t+1} \le \left(1-\sqrt{\frac{\alpha}{L_t}}\right)E_t,
\]
valid at any iteration where we do not restart (i.e., when $y_t\in\mathcal K_1$ and the Lyapunov update is applied).
Let $0=\tau_0<\tau_1<\cdots<\tau_R\le T$ denote the restart times (epoch starts), and set $\tau_{R+1}:=T+1$.
On each epoch $[\tau_j,\tau_{j+1})$, unrolling the above inequality gives
\[
E_{\tau_{j+1}}
\le
\left(
\prod_{t=\tau_j}^{\tau_{j+1}-1}\left(1-\sqrt{\frac{\alpha}{L_t}}\right)
\right)
E_{\tau_j}^+,
\]
where $E_{\tau_j}^+$ denotes the value of $E$ \emph{after} the restart initialization at time $\tau_j$
(and $E_{\tau_0}^+=E_0$).
At each restart time $\tau_j$ with $j\ge 1$, we have $z_{\tau_j}=g_{\tau_j}$ and by $\alpha$-strong convexity
$\mathcal J(g_{\tau_j})-\mathcal J^\star \ge \frac{\alpha}{2}\|g_{\tau_j}-g^\star\|^2$, hence
\[
E_{\tau_j}^+
= \mathcal J(g_{\tau_j})-\mathcal J^\star+\frac{\alpha}{2}\|g_{\tau_j}-g^\star\|^2
\le 2\big(\mathcal J(g_{\tau_j})-\mathcal J^\star\big)
\le 2E_{\tau_j}^-,
\]
where $E_{\tau_j}^-$ denotes the value of $E$ just before restarting (same $g_{\tau_j}$, previous $z_{\tau_j}$).
Iterating over epochs and using $E_{t}\ge \mathcal J(g_t)-\mathcal J^\star$ yields the global bound
\[
\mathcal J(g_{T+1})-\mathcal J^\star
\le
2^{R}\left(\mathcal J(g_0)-\mathcal J^\star+\frac{\alpha}{2}\|z_0-g^\star\|^2\right)
\prod_{t=0}^{T}\left(1-\sqrt{\frac{\alpha}{L_t}}\right).
\]
In particular, since $z_0=g_0$ at initialization, we have $\|z_0-g^\star\|=\|g_0-g^\star\|$. Moreover, $L_t\le \bar L$, where $\bar{L} = O(C_{\text{bound}}/\varepsilon)$ using that all iterations are in $\mathcal K_1$. We thus conclude
\[
\mathcal J(g_{T+1})-\mathcal J^\star
\le
2^{R}\left(\mathcal J(g_0)-\mathcal J^\star+\frac{\alpha}{2}\|g_0-g^\star\|^2\right)
\left(1-\sqrt{\frac{\alpha}{\bar L}}\right)^{T+1}.
\]

\end{proof}

\section{Towards Optimal Adaptive Step Sizes}\label{app::towards}

Optimal rates for ASGD are obtained with a learning rate of $\mathcal{O}\left(\frac{1}{L + \beta k}\right)$ for $L$-smooth and $\beta$-strongly convex functions~\cite{stich2019unified}. In this case, the rate is:
\begin{align*}
    \mathbb{E}f(\bar{\mathbf{x}}_T) - f^{\star} + \beta \mathbb{E}\|\mathbf{x}_{T+1} - \mathbf{x}^{\star}\|^2 = \tilde{\mathcal{O}} \left( L R^2 \exp \left[ -\frac{\beta T}{2L} \right] + \frac{\sigma^2}{\beta T} \right)\ .
\end{align*}
In our context, if we were near the minimum, this would lead to a complexity of $\mathcal{O}(\rho_2/\varepsilon)$ to eliminate the transient exponential term. This results in a total complexity of $\mathcal{O}(n\rho_2/ T)$ for any $\varepsilon$, whereas a global bound would require $\mathcal{O}(\rho_2n/\varepsilon)$ to handle this exponential term.

In a broader context, the motivation for employing a learning rate of the form $(1/L_t + \beta k)^{-1}$ and leveraging adaptive smoothness was previously investigated by~\cite{malitsky2019adaptive}. There, the authors proposed a local estimator of smoothness as a heuristic; however, this approach did not yield theoretical acceleration and resulted in worse constants.

While we will not be able to strictly use or prove this optimal rate here, we remark that for $\mathcal{J}$, another natural choice leads to a similar schedule. A classical schedule for $\gamma_t \propto 1/t$ sets the constant as the inverse of the strong convexity, i.e., $\gamma_t^{sc} = \frac{n}{C\rho_2}$. Concurrently, assuming $C/n = \beta_{\min} = \beta_{\max}$ for clarity, the optimal learning rate from~\cite{stich2019unified} near the optimum would be:
\[
    \gamma_t^{\text{opt}} = \frac{1}{1/L(\mathbf{g}^\star) + \beta t} \simeq \frac{n}{C(1/\varepsilon + \rho_2t)}\ .
\]
We observe that, in this case, the two learning rates differ simply by an offset of $1/\varepsilon$.

\begin{figure}[H]
    \centering
    \includegraphics[width=0.4\linewidth]{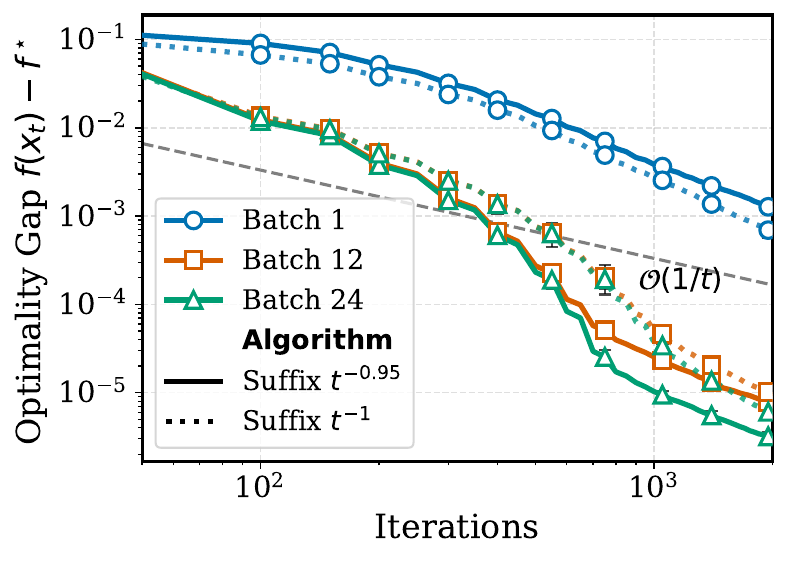}
    \caption{Comparison of ASGD convergence rates using the optimal learning rate schedule $\gamma_t^{\text{opt}} = (1/L(\mathbf{g}^\star) + \beta t)^{-1}$ versus the polynomial decay $\gamma_t = (1/L(\mathbf{g}^\star) + \beta t^b)^{-1}$ with $b=0.95$. Here $ \beta_{\min} = \beta_{\max} = 1/n$ with $n = 2000$.  Both methods employ suffix averaging (averaging the last half of iterates), as recommended in~\cite{stich2019unified}. We display results for minibatch sizes of 1, 12, and 24, averaged over 20 independent runs. We observe that both learning rate schedules perform equally well. Variance is omitted from the plot as it is negligible.}
    \label{fig:suffix_comparison}
\end{figure}

Although we do not provide a proof of the true adaptivity of our SGD to local strong convexity for the $\propto 1/t$ learning rate, we provide next a proof of ASGD for any learning rate $\gamma_t \propto 1/t^b$ where $b \in (0,1)$. We select $b < 1$ to ensure the convergence of all moments of our SGD scheme. This allows us to demonstrate a mild adaptivity phenomenon by taking $\gamma_t \simeq \frac{n}{C(1/\varepsilon + \rho_2t^b)}$ with $b$ very close to $1$. In contrast, we would not be able to prove this adaptivity with $\gamma_t \propto 1/t$ and it would lead to worse constants in $n$ and $1/\varepsilon$.

\section{Additional Experimental Results}
\label{app:experiments}

\textbf{Comparison with Sinkhorn.} We benchmark our Adaptive NAG (ANAG) method against the Translation Invariant (TI) Sinkhorn algorithm \citep{sejourne2022faster} on a standard color transfer task. We select 20 random image pairs from the CIFAR-10 dataset, treating pixels as point clouds in RGB space ($n=4096$). Both solvers are run with regularization $\varepsilon=0.01$ and marginal penalty $\rho=10$.

It is important to note that ANAG minimizes the $\KL-\chi^2$ semi-dual, while TI-Sinkhorn solves the standard $\KL-\KL$ formulation. However, in this regime, the resulting optimal couplings and transport costs are nearly identical, justifying a direct comparison of their convergence profiles. We establish a ground truth value $f^\star$ by running each solver for $20,000$ iterations and report the median relative objective gap $(f(x_t) - f^\star)/|f^\star|$ in Figure~\ref{fig:cifar_benchmark}.

\begin{figure}[h]
    \centering
    \includegraphics[width=0.5\linewidth]{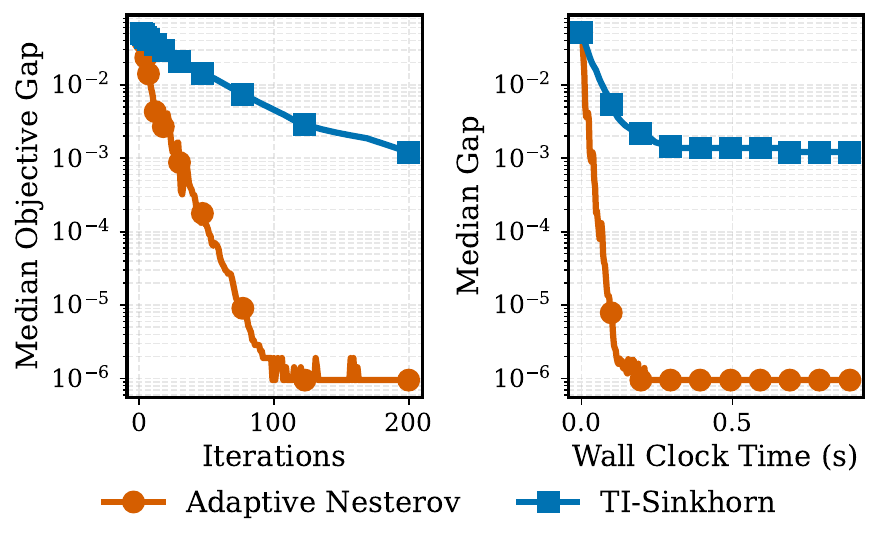}
    \caption{\textbf{ANAG vs. TI-Sinkhorn.} Median objective gap convergence on CIFAR-10 color transfer tasks ($20$ pairs of size $n = 4096$). ANAG demonstrates competitive convergence rates compared to the TI-Sinkhorn algorithm (solving $\text{KL}-\text{KL}$), validating the efficiency of the adaptive scheme on standard semi-discrete tasks.}
    \label{fig:cifar_benchmark}
\end{figure}

\end{document}